\documentclass{amsart}
\usepackage{amsmath,amsfonts,amsthm,amssymb}
\usepackage[alphabetic]{amsrefs}
\usepackage{hyperref}

\newtheorem{lem}{Lemma}[section]
\newtheorem{tw}[lem]{Theorem}

\newtheorem{cor}[lem]{Corollary}
\newtheorem{prop}[lem]{Proposition}

\newtheorem*{twA}{Theorem A}
\newtheorem*{twB}{Theorem B}
\newtheorem*{twC}{Theorem C}
\newtheorem*{twD}{Theorem D}
\newtheorem*{twE}{Theorem E}

\newcommand {\dow}{\begin{proof}}
\newcommand {\kon}{\end{proof}}

\newcommand {\mr}{\mathrm}
\newcommand {\lk}{\left\{}
\newcommand {\rk}{\right\}}
\newcommand {\lan}{\langle}
\newcommand {\ran}{\rangle}
\newcommand {\rems}{\noindent {\bf Remarks. }}
\newcommand {\rem}{\noindent {\bf Remark. }}

\theoremstyle{definition}
\newtheorem{de}[lem]{Definition}

\newcommand {\wt}{\widetilde}

\newcommand {\ov}{\overline}
\newcommand {\la}{\langle}
\newcommand {\ra}{\rangle}

\usepackage{pictexwd,dcpic,epsf}
\usepackage{graphicx}

\begin{document}

\title{A combinatorial non-positive curvature I: weak systolicity}
\author{Damian Osajda}
\date{\today}
\thanks{Partially supported by MNiSW
grant N N201 541738.}
\address{Universit\"at Wien, Fakult\"at f\"ur Mathematik\\
Oskar Morgensternplatz 1, 1090 Wien, Austria\\
and}
\address{(on leave from) Instytut Matematyczny,
Uniwersytet Wroc\l awski\\
pl. Grunwaldzki 2/4,
50--384 Wroc{\l}aw, Poland}
\email{dosaj@math.uni.wroc.pl}

\begin{abstract}
We introduce the notion of weakly systolic complexes and groups, and initiate regular studies of them. Those are simplicial complexes with nonpositive-curvature-like properties and groups acting on them geometrically. We characterize weakly systolic complexes as simply connected simplicial complexes satisfying some local combinatorial conditions. We provide several classes of examples  ---  in particular systolic groups and CAT(-1) cubical groups are weakly systolic.
We present applications of the theory, concerning Gromov hyperbolic groups, Coxeter groups and systolic groups.
\end{abstract}

\maketitle

\section{Introduction}
\label{intro}
\subsection{Overview}
In recent decades exploration of various notions of nonpositive curvature (NPC) became one of the leading subjects in geometry and in related fields. In particular, in geometric group theory a significant role is played by studies of Gromov hyperbolic and CAT(0) spaces and groups; cf.\ e.g.\ \cites{BrHa,ECH,Gr-hg}. Objects arising in this way have nice algorithmic properties and some applications (beyond the pure mathematics) for them have been found. Moreover, the framework allows to treat at the same time many classical groups, e.g.\ fundamental groups of nonpositively curved manifolds.

The combinatorial approach to nonpositive curvature, i.e.\ studying groups acting on NPC (nonpositively curved) complexes, yields a rich source of NPC groups with interesting and often unexpected properties. A good example in dimension two is the classical theory of small cancellation groups that emerged initially in the combinatorial group theory. As for higher dimensional objects, at the moment the two leading subjects of interests are CAT(0) cubical complexes and groups (cf.\ \cites{BrHa,Gr-hg}), and systolic complexes and group (cf.\ \cites{JS1,Ha}). The combinatorial nature allows in those cases to get quite profound insight into their structure and thus to obtain strong results. For example, there is a big effort to ``cubulate" many classical groups, in order to get e.g.\ various separability results.
On the other hand many constructions of NPC complexes have been developed recently, giving us useful tools for producing examples of NPC groups with interesting properties; cf.\ \cite{DO,Ha,JS1,O-chcg,OS}.\smallskip

In this paper we introduce \emph{weakly systolic complexes} and initiate regular studies of them and of groups acting on them geometrically, i.e.\ \emph{weakly systolic groups}.
Weakly systolic complexes are simplicial complexes possessing many nonpositive-curvature-like properties. They can be characterized as simply connected simplicial complexes satisfying some local combinatorial conditions.
This forms a new class of \emph{``combinatorially nonpositively curved''} complexes and groups.
Systolic complexes and groups (cf.\ \cite{JS1}) are weakly systolic. Groups acting geometrically on simply connected cubical complexes with flag-no-square links (i.e.\ on simply connected locally $5$--large, or CAT(-1), cubical complexes) are weakly systolic. We describe also other classes of examples. We provide some applications of ideas and results from this paper, to the theory of systolic groups, Coxeter groups, and Gromov hyperbolic groups. Other results concerning the theory of weakly systolic complexes and groups can be found in e.g.\ \cites{BCC+,OCh, O-chcg, OS}.

\subsection{Main results}
In Section~\ref{prosdn} we define \emph{weakly systolic} complexes as flag simplicial complexes satisfying a global combinatorial property $SD_n$  ---  of \emph{simple descent on balls}.
Immediately from that definition the following important property of weakly systolic complexes is derived there.

\begin{twA}[cf.\ Proposition~\ref{sdncontr} in the text]
  Finitely dimensional weakly systolic complexes are contractible.
\end{twA}

On the way we provide an equivalent definition of the $SD_n$ property that is more convenient for applications.

In Section~\ref{locglo} we define a local combinatorial condition $SD_2^{\ast}$ and prove the following important result.

\begin{twB}[Theorem~\ref{logl} in the text]
  Let a simplicial complex satisfy the condition $SD_2^{\ast}$. Then its universal
cover is weakly systolic.
\end{twB}

Together with Theorem A, the latter result form a version of the Cartan-Hadamard theorem, and is the fundament of the whole theory.

In Section~\ref{convex} we derive some convexity properties of balls in weakly systolic complexes. Those results provide technical tools used in other places of the paper and beyond. Then, in Section~\ref{ex}, we provide important examples of weakly systolic complexes, and of \emph{weakly systolic groups}, i.e.\  groups acting geometrically on weakly systolic complexes. By the way, we develop a useful tool  ---  the \emph{thickening} of cell complexes; cf.\ Subsection~\ref{cc-1}.

\begin{twC}[cf.\ Section~\ref{ex}]
  The following classes of groups are weakly systolic:
  \begin{itemize}
    \item Systolic groups;
    \item CAT(-1) cubical groups;
    \item Uniform lattices of right-angled hyperbolic buildings.
  \end{itemize}
\end{twC}

It is worth to notice that some CAT(-1) cubical groups provide examples of weakly systolic groups that are not systolic (cf.\ Remarks after Corollary~\ref{cat-1cc}), despite the fact that most of weakly systolic techniques are quite ``systolic".
We find also other examples of weakly systolic groups, e.g.\ some subgroups of such groups.

In Section~\ref{neg} we study ``negative curvature" in the context of weak systolicity. In Theorem~\ref{noflat} we show that for weakly systolic groups Gromov hyperbolicity is equivalent to the non-existence of isometrically embedded flats. Then (Subsection~\ref{nega}) we define a local combinatorial condition $SD_2^{\ast}(7)$, an analogue of a negative curvature. The following result justifies the analogy.

\begin{twD}[Theorem~\ref{hyp} in the text]
  The universal cover of an $SD_2^{\ast}(7)$ complex is Gromov hyperbolic.
\end{twD}

In particular, we obtain a nice description of Gromov boundaries of negatively curved weakly systolic complexes (see Subsection~\ref{nega} for the details).

\begin{twE}[Theorem~\ref{bdryinv} in the text]
  The Gromov boundary of a weakly systolic $SD_2^{\ast}(7)$ complex is homeomorphic to the associated inverse limit of spheres.
\end{twE}

We also show (Corollary~\ref{qcsbgp}) that quasiconvex subgroups of weakly systolic $SD_2^{\ast}(7)$ groups are weakly systolic $SD_2^{\ast}(7)$ groups as well.

In Section~\ref{sec-lasf} we study \emph{weakly systolic complexes with $SD_2^{\ast}$ links} and groups acting on them geometrically. These are weakly systolic complexes and groups whose asymptotic behavior resembles a lot the one of systolic counterparts. In particular we show that finitely presented subgroups of such torsion-free groups are of the same type (Theorem~\ref{fps})  ---  this is an analogue of a corresponding systolic result of D.\ Wise. Then we prove various properties of asymptotic asphericity; see  Theorem~\ref{coninfth} and Theorem~\ref{shatw}.

\subsection{Motivations and applications}
Our initial motivation was to provide a set of local combinatorial conditions on a simplicial complex that guarantee, that the universal
cover of such complex exhibits various nonpositive-curvature-like properties. Moreover, we would like a new theory to include the existing
examples of such ``combinatorial nonpositive curvature". In particular we would like that systolic groups (cf.\ \cite{JS1}) and CAT(0)
cubical groups (i.e.\ groups acting geometrically on simply connected cubical complexes with flag links; cf.\ \cite{BrHa}) act on complexes
from the new class. The question of such a ``unification" (of systolic and CAT(0) cubical theories) has been raised number of times by
various people. Let us mention here Mladen Bestvina, Fr\' ed\' eric Haglund, Tadeusz Januszkiewicz an
d Jacek \' Swi\polhk{a}tkowski, who suggested this problem to us. 

The goal of the unification is achieved only partially via the theory of weakly systolic complexes. Systolic groups and CAT(-1) cubical
groups are weakly systolic, but we conjecture (cf.\ Section \ref{othex}) that some CAT(0) cubical groups are not weakly systolic.
Nevertheless, the theory provides powerful tools for studying a large class of groups. In fact most of the techniques developed in this
paper (and in \cites{OCh, O-chcg, OS}) are quite ``systolic" in the spirit. And, surprisingly, they can be used in very (i.e.\
asymptotically) ``non-systolic" cases; cf.\ e.g.\ Section \ref{cc-1}. Several applications of the theory has been already found (see below). And the main idea of the \emph{simple
descent on balls} (i.e.\ the $SD_n$ property from Definition \ref{sdn(A)new}), being the core of the combinatorial nonpositive curvature
in our approach, is being developed by us to provide a full unification theory. Actually, we believe that other theories
in the same spirit can provide a framework to treat various classes of groups, including many classical ones. On the other hand, from the topological viewpoint, our main results --- Theorem A and Theorem B --- provide local combinatorial conditions on a simplicial complex guaranteeing its asphericity. Not many such conditions are known, the well known being CAT(0) property and systolicity; for other ones see e.g.\ \cite{BCC+}.

Similarly as systolic complexes, weakly systolic complexes are \emph{flag}, which means that they are determined by their $1$--skeleta.
Graphs being $1$--skeleta of systolic complexes are \emph{bridged graphs}. This class of graphs has been intensively studied within the metric and algorithmic graph theory for last decades; cf.\ e.g.\ \cites{AnFa,BaCh,Ch-class,Ch-bridged,Ch-CAT} and references therein.
Also $1$--skeleta of CAT(0) cubical complexes, known as \emph{median graphs}, has been explored in graph theory since a long time, before the recent rise of interest in the corresponding complexes; cf.\ e.g.\ \cites{BaCh,Ch-CAT} and references therein.
To the contrary, \emph{weakly bridged graphs} (see Subsection~\ref{wbg}) being $1$--skeleta of weakly systolic complexes, were not
investigated in metric graph theory before. They appearance seems quite natural in this context, though. And introducing them leads to
many interesting structural problems; some of them are treated in \cites{OCh,BCC+} (and some are solved there using tools from the
current paper).
\medskip

The techniques and results presented in this article found already few applications, and we are working on further ones. The proof of
Cartan-Hadamard theorem (Theorem B above) given in Section~\ref{locglo}, presents a general scheme for proving
results of this type. It can be used in other NPC cases (we use it, even following closely the notations, in \cite{BCC+}, and for some
more general complexes, in progress) and beyond (see \cite{CCO} where we use it in a case of ``positive curvature"). Moreover, such a
proof allows to show developability of complexes of groups with certain (in particular ``weakly systolic") local developments properties
(work in progress). This in turn leads to new constructions of combinatorially nonpositively curved groups --- fundamental groups of the
corresponding complexes of groups (in progress).

Another important tool introduced in this paper is thickening. The \emph{thickening} of a cell complex is a simplicial complex whose simplices correspond to cells of the former complex --- see Section~\ref{ex} for details. The use of
thickening reduces studies of a complicated complex to studying the corresponding simplicial complex, whose properties (e.g.\
homotopical properties) resemble the ones of the original object. In Section~\ref{ex} we use thickening to show that CAT(-1) cubical
groups and lattices of right-angled hyperbolic buildings are weakly systolic (Theorem C above). Those results form the core of new
constructions of Gromov hyperbolic groups provided in \cite{O-chcg,OS}.
The technique of thickening may be used in a case of quite general complexes (an instance being the thickening of buildings from Subsection~\ref{build}), and we believe that it is a very useful tool.

Furthermore, the thickening together with the combinatorial negative curvature (Theorem D) introduced in Section~\ref{neg}, allows one to e.g.\
describe nicely Gromov boundaries of some classical groups; see Theorem E and Remark after Theorem~\ref{bdryinv}. Such description
gives a good insight into the structure of the boundary as a limit of polyhedra. We believe that this can be helpful in exploring various structures on such boundaries.

Weakly systolic groups with $SD_2^{\ast}$ links introduced in Section~\ref{sec-lasf} are groups whose asymptotic behavior resembles a lot the one of systolic groups. The latter groups lead to examples of highly dimensional groups with interesting asphericity properties; see \cites{JS2,O-ciscg,O-ib7scg,Sw-propi}. Results from Section~\ref{sec-lasf} allowed us in \cite{OS} to provide
new constructions of groups with the same properties. On the other hand, using the technique of thickening and tools from Section~\ref{neg} we can provide the first examples of non-systolic highly dimensional ``asymptotically aspherical" groups (in progress).\medskip

Finally, let us note that in the current paper we provide few more elementary and/or (at the same time) more general
proofs of some already known results concerning systolic complexes. Those are e.g.: the proof of contractibility of systolic complexes provided in Subsection~\ref{syst} (see Corollary~\ref{systcontr} and Remark afterwards); the proof of Przytycki's theorem about
flats versus hyperbolicity (Theorem~\ref{noflat}).

\subsection{Acknowledgments}
I am grateful to Mladen Bestvina, Victor Chepoi, Tade\-usz Januszkiewicz and Jacek \' Swi\polhk atkowski for stimulating and helpful discussions.
I thank the most Fr\' ed\' eric Haglund. The whole theory was inspired by his ideas about combinatorial nonpositive curvature. I am grateful for his continuous encouragement and his support.

\tableofcontents

\section{Preliminaries}
\label{prel}

\subsection{Simplicial complexes}
Let $X$ be a simplicial complex. We do not usually distinguish between a simplicial complex and its geometric realization. The $i$--skeleton of $X$ is denoted by $X^{(i)}$. A subcomplex $Y$ of $X$ is \emph{full} if
every subset $A$ of vertices of $Y$ contained in a simplex of $X$, is
contained in a simplex of $Y$.
For a subcomplex $A$ of $X$, by $\langle A \rangle$ we denote the \emph{span} of $A$, i.e.\ the smallest full subcomplex of $X$ containing $A$. If $A$ is the set of vertices $v_1,v_2,\ldots$, then we write $\langle v_1,v_2,\ldots \rangle$ for $\langle A \rangle$. Thus ``$\langle A \rangle \in X$" or ``$\langle v_1,v_2,\ldots \rangle \in X$" mean that the corresponding sets span a simplex in $X$.
If vertices $v\neq w$ span an edge $\langle v,w \rangle$, we denote this edge simply by $vw$.
We write $v\sim v'$ (respectively $v\nsim v'$) if $\langle v,v' \rangle \in X$ (respectively $\langle v,v' \rangle \notin X$). Moreover, we write $v\sim v_1,v_2,\ldots$ (respectively $v\nsim v_1,v_2,\ldots$) when $v\sim v_i$ (respectively $v\nsim  v_i$) for $i=1,2,\ldots$.

A simplicial complex $X$ is \emph{flag} whenever every finite set of vertices of $X$ joined pairwise by edges in $X$, is contained in a simplex of $X$.
A \emph{link} of a simplex $\sigma$ of $X$ is a simplicial complex $X_{\sigma}=\lk \tau \; | \; \tau \in X \; \& \; \tau \cap \sigma=\emptyset \; \& \; \langle \tau \cup \sigma \rangle \in X \rk$. Let $k\geqslant 4$. A \emph{$k$--cycle} $(v_0,\ldots,v_{k-1},v_0)$ is a triangulation of a circle consisting of $k$ edges $v_iv_{i+1\; (\mr{mod}\; k)}$ and $k$ vertices: $v_0,\ldots,v_{k-1}$.

If it is not stated otherwise we consider the \emph{(combinatorial) metric} on the set of vertices $X^{(0)}$, defined as the number of edges in the shortest path connecting given two vertices; it is usually denoted by $d(\cdot, \cdot)$. For subcomplexes $A,B$ of $X$ we set $d(A,B)=\inf \lk d(v,w)|\; v\in A^{(0)}, \; w\in B^{(0)} \rk$.
For a subcomplex $Y$ of a flag simplicial complex $X$, by $B_i(Y,X)$ (respectively, $S_i(Y,X)$) we denote the \emph{(combinatorial) ball}
(respectively, \emph{sphere}) of radius $i$ around $Y$, i.e.\ the subcomplex of $X$ spanned by the set of vertices at distance at most $i$ (respectively, exactly $i$) from the set of vertices of $Y$. If it does not lead to a confusion we write $B_i(Y)$ or $S_i(Y)$ instead of $B_i(Y,X)$ or $S_i(Y,X)$.

By $X'$ we denote the first barycentric subdivision of a simplicial complex $X$. For a simplex $\sigma$ of $X$, by $b_{\sigma}$ we denote a vertex of $X'$ corresponding to $\sigma$.

\subsection{Cell complexes}
\label{cellc}

For the definition and basics concerning cell complexes we follow \cite[Appendix A]{Davi-b}.
\medskip

A \emph{cell complex} is a collection of convex polytopes, called \emph{cells}, such that each sub-cell is a cell, and
the intersection of every two cells is again a (possibly empty) cell. A simplicial complex is a cell complex whose all cells are simplices. A \emph{cubical complex} is a cell complex whose all cells are cubes $[0,1]^n$.

Consider a cell $c$ of a cell complex $X$. Let $\mathcal P$ be the poset of nonempty cells strictly containing
$c$. The cell complex whose poset of cells is isomorphic to $\mathcal P$ is called the \emph{link} of $c$ in $X$, and is denoted by
$X_c$. A cell complex is \emph{simple} if all its links are simplicial complexes.

\begin{de}[loc.\ $k$--large]
 \label{l4lcc}
Let $k$ be a natural number $\geqslant 4$. A simple cell complex $Y$ is
\emph{locally $k$--large} if every its link is $4$--large.
\end{de}

\rem By a lemma of Gromov \cite{Gr-hg} a locally $4$--large (respectively, locally $5$--large) cubical
complex is a locally CAT(0) (respectively, locally CAT(-1)) space when equipped with a piecewise Euclidean (respectively hyperbolic) metric in
which every cube is isometric to the standard cube in $\mathbb E^n$ (respectively, in $\mathbb H^n$).

\section{Property $SD_n$ and weak systolicity}
\label{prosdn}

In this section we introduce central objects of this paper: the $SD_n$ property and the weakly systolic complexes. We also prove that weakly systolic complexes are contractible.

\begin{de}[Property $SD_n(A)$]
\label{sdn(A)new}
Let $X$ be a flag simplicial complex, let $A$ be its subcomplex and let $n\in \{0,1,2,3,\ldots\}$.
We say that
$X$ \emph {satisfies the property $SD_n(A)$} (simple descent on balls of radii at most $n$ around $A$) if for every $i=1,2,\ldots,n$ the following
condition holds. For every simplex $\sigma \in S_{i+1}(A)$ the intersection
$\pi_A(\sigma)=X_{\sigma}\cap B_i(A)$ is a non-empty simplex.
\end{de}

For technical reasons the following, weaker variation of the definition above will be useful.

\begin{de}[Property $\wt {SD}_n(A)$]
\label{sdn(A)}
Let $X$ be a flag simplicial complex, let $A$ be its subcomplex and let $n\in \{0,1,2,3,\ldots\}$.
We say that
$X$ \emph {satisfies the property $\wt{SD}_n(A)$} if for every $i=1,2,\ldots,n$ the following two conditions hold.
\begin{description}
  \item[({\bf E}) \rm(edge condition)] For every
edge $e \in S_{i+1}(A,X)$ the intersection
 $X_e\cap B_i(A,X)$ is non-empty.
  \item[({\bf V}) \rm{(vertex condition)}] For every vertex $w \in S_{i+1}(A,X)$ the intersection
 $X_w\cap B_i(A,X)$ is a single simplex.
\end{description}
\end{de}

\begin{de}[Weakly systolic]
\label{sdn}
A simplicial complex $X$ is \emph {weakly systolic} if it is flag and if it satisfies the property $\wt{SD}_n(v)$ for every natural number $n$ and for every vertex $v\in X$.
\end{de}

\begin{lem}[SD vertices vs.\ simplices]
\label{sd vert}
Let $A$ be a subcomplex of a flag simplicial complex $X$ and let $\sigma$ be a simplex
of $X$. Assume that $X_{\sigma}\cap A \neq \emptyset$.
Assume that $X_v\cap A$ is a single simplex, for some vertex $v$ of $\sigma$.
Then
$X_{\sigma}\cap A$ is a single simplex.
\end{lem}
\begin{proof}
It follows easily from the flagness and the fact that $X_{\sigma}\subset X_v$.
\end{proof}

The following lemma
shows that in the Definition \ref{sdn} we can use the (more intuitive) property $SD_n(v)$ instead of $\wt{SD}_n(v)$.

\begin{lem}[$SD_n$ vs. $\wt {SD}_n$]
\label{sd simpl}
Let $X$ be a weakly systolic complex.
Then $X$ satisfies the property $SD_n(v)$ for every natural number $n$ and for every vertex $v\in X$.
\end{lem}
\begin{proof}
By Lemma \ref{sd vert} it is enough to prove that the intersection
$X_{\sigma}\cap B_i(v,X)$ is non-empty, for every $i$, every $v$ and for every simplex $\sigma$ in $S_{i+1}(v)$.

We show this by induction on $m=\mr {dim}(\sigma)$. By weak systolicity it is true for
$m=0,1$.
Assume this is proved for $m$. We show that it holds for
$m+1$.
Let $\sigma$ be an $(m+1)$--dimensional simplex in $S_{i+1}(v)$.
Let $\sigma=\lan \sigma' \cup w \ran$ where $w\notin \sigma'$. Let $w'$ be a vertex of $\sigma'$.
By inductive assumption there exist vertices $z,z'\in B_i(v)$ such
that $\langle\sigma', z'\rangle,\langle w',w,z\rangle\in X$. Let $\sigma''$ be a maximal
simplex in $S_{i+1}(v)$ joinable with both $z,z'$. Observe that
$w'\in \sigma''$.
We consider separately three cases.
\medskip

\noindent
\emph{Case 1.} If $w\in \sigma ''$ then $\langle\sigma, z'\rangle\in X$ and the
lemma follows.
\medskip

\noindent
\emph{Case 2.} If $\sigma' \subseteq \sigma''$ then $\langle\sigma, z\rangle\in X$
and the lemma follows.
\medskip

\noindent
\emph{Case 3.} If neither Case 1 nor 2 hold then there exists a
vertex $w''\in \sigma' \setminus \sigma''$. Analyzing the cycle
$(w'',w,z,z',w'')$ we have, by weak systolicity, that either $w''\sim z$ or $w\sim z'$. In both cases we get contradiction.
\end{proof}

{\bf Remark.} Let $A$ be a subcomplex of a flag simplicial complex
$X$. Observe that if
$X$ satisfies the property $\wt {SD}_n(A)$ then
$X_{\sigma}\cap B_n(A)$ can be empty for a simplex
$\sigma \in
S_{n+1}(A)$.
\medskip

Now we prove that weakly systolic complexes are contractible. We follow the
approach of \cite[Section 5]{JS1} (for the case of systolic
complexes). Let $X$ be a complex of dimension $d<\infty$
and let $v$ be its vertex.
Assume that $X$ satisfies the property $SD_n(v)$ for every $n$.
For $l=0,1,\ldots,d-1$ we define (compare
\cite[Section 5]{JS1}) the sets $Q^l=B_{i}(v)\cup \bigcup \lk
\lan \sigma \cup \pi_v(\sigma) \ran \;| \; \; \sigma \in S_{i+1}(v) ,\; \;
\mr{dim}(\sigma)\leqslant l
\rk$. We set also $Q^{-1}=B_i(v)$. Observe that $Q^{d-1}=B_{i+1}(v)$.

\begin{lem}
 \label{defretr}
There exists a (canonical) deformation retraction
$r^l\colon Q^l\to Q^{l-1}$.
\end{lem}
\begin{proof}
The deformation retraction
is defined by the property that for every $l$--simplex $\sigma \in
S_{i+1}(v)$ the restriction $r^l|_{\lan \sigma \cup \pi_v(\sigma)\ran}\colon \lan \sigma \cup
\pi_v(\sigma) \ran \to \lan (\partial \sigma)\cup \pi_v(\sigma)\ran$ is a standard deformation retraction
of a simplex.
\end{proof}

\begin{prop}[$SD_n(v)$ implies contractibility]
\label{sdncontr}
Let $X$ be a finite dimensional simplicial complex satisfying $SD_n(v)$ property for some vertex $v$ and for every natural number $n$. Then $X$ is contractible.
\end{prop}
\begin{proof}
For every $i$ we define a deformation retraction $r_i\colon
B_{i+1}(v)\to B_i(v)$ as the composition $r^0 \circ r^1 \circ \cdots
\circ r^{d-1} $. Contractibility of $X$ follows by applying
compositions of maps $r_i$ for different $i$'s.
\end{proof}

\rem Observe that the above approach yields in fact collapsibility of complexes in question. This is a property
stronger than contractibility, but we do not need to use it in this paper. In \cite{OCh} we show that weakly systolic
complexes satisfy even stronger property --- their $1$--skeleta are dismantlable graphs.

\subsection{Weakly bridged graphs}
\label{wbg}

For systolic complexes, their $1$--skeleta belong to the class of \emph{bridged} graphs. Here we define the corresponding class of weakly bridged graphs. Note, that $1$--skeleta of weakly systolic complexes do not contain infinite cliques, since we consider only finite simplices.
For arbitrary graphs one can define balls and links of vertices analogously as we did above in Section~\ref{prel} (e.g.\ link is a graph spanned by neighbors of a given vertex). Thus the definition of the property $\wt {SD}_n(v)$ makes sense.

\begin{de}[Weakly bridged graph]
\label{defwb}
A graph $\Gamma$ is \emph {weakly bridged} if it satisfies the property $\wt{SD}_n(v)$ for every natural number $n$ and for every vertex $v\in X$.
\end{de}

\rem In \cite{OCh} we show some structural properties of weakly bridged graphs relating them to well known classes of graphs.

\section{Local to global}
\label{locglo}
The aim of this section is to show that some local conditions on a complex $X$ imply that the universal cover of $X$ is weakly systolic and thus, by Proposition \ref{sdncontr}, contractible.
This result is an analogue of the Cartan-Hadamard theorem for nonpositively curved spaces (cf.\ \cite[Chapter II.4]{BrHa}) and is the heart of the whole theory developed in our paper.

A rough idea is that if $X$ satisfies the $SD_2$ property then its universal cover
satisfies the conditions $SD_n$, for all $n$. However, we will prove a stronger result where instead of the property $SD_2$ we consider its version --- the property $SD_2^{\ast}$. It is much more useful for providing examples --- cf.\ Section \ref{sec-lasf} and \cite{O-chcg}.

\begin{de}[Wheels]
 \label{wheels}
Let $X$ be a simplicial complex and let $k\geqslant 4$ be a natural number.
A \emph{$k$--wheel} $(v_0;v_1,\ldots,v_k)$, where $v_i$'s are vertices of
$X$, is a full subcomplex of $X$ such that $(v_1,\ldots,v_k,v_1)$ is a full
$k$--cycle and $v_0$ is joinable with $v_i$, for $i=1,\ldots,k$.

A \emph{$k$--wheel with a pendant triangle} $(v_0;v_1,\ldots,v_k;t)$ is a
subcomplex of $X$ being the union of a $k$--wheel $(v_0;v_1,\ldots,v_k)$ and a triangle
$\langle v_1,v_2,t\rangle\in X$, with $t\neq v_i$, $i=0,1,\ldots,k$.
\end{de}

\begin{figure}[t]
\begin{center}
\scalebox{0.5}{\includegraphics{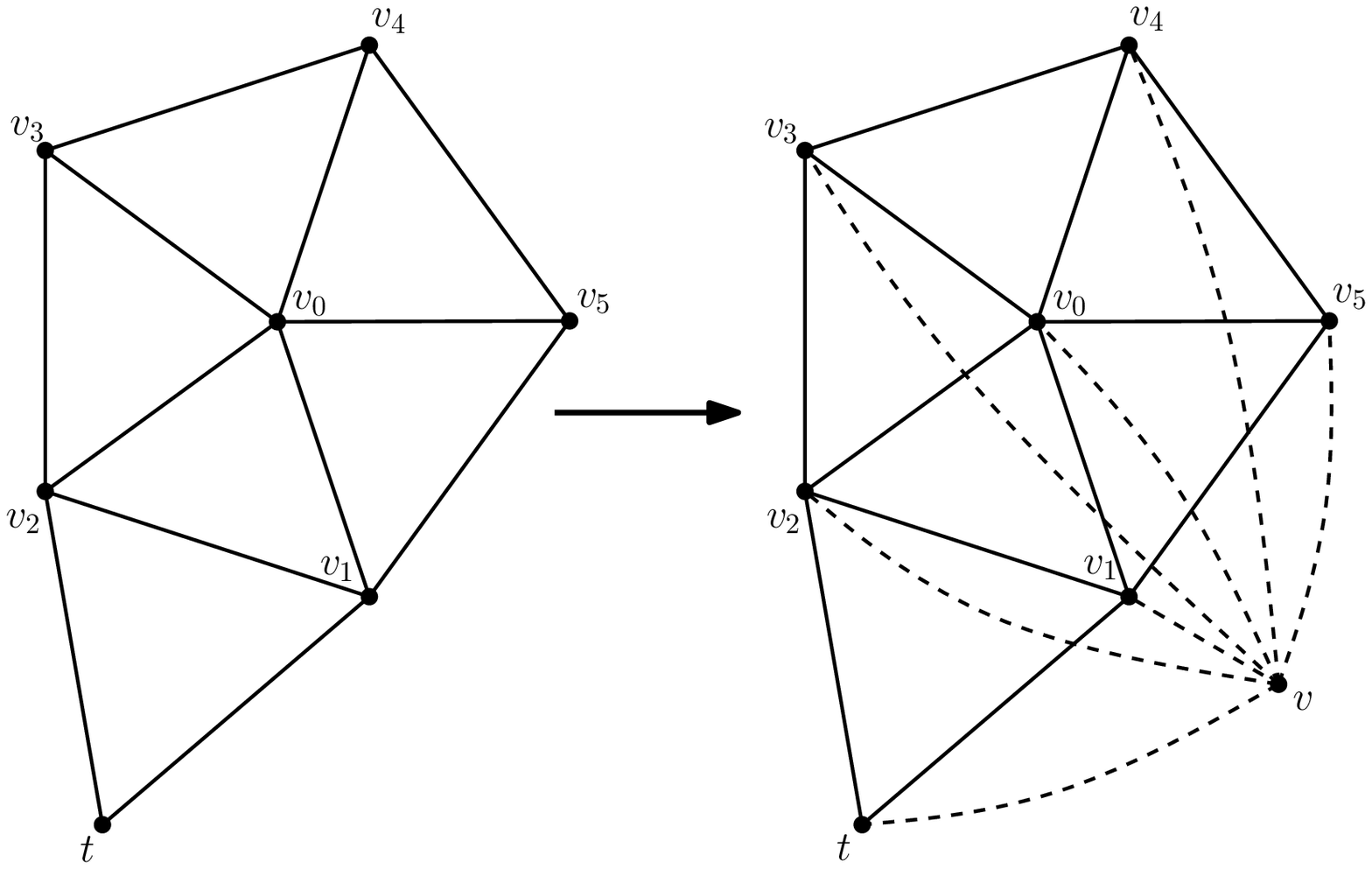}}%
\end{center}
\caption{$SD_2^{\ast}$ property}
\label{SD2}
\end{figure}

\begin{de}[$SD_2^{\ast}$ property]
\label{sd2*}
A flag simplicial complex $X$ satisfies the $SD_2^{\ast}$ property if the following two conditions hold.

(a) $X$ does not contain $4$--wheels,

(b) for every $5$--wheel with a pendant triangle $\widehat W$ in $X$, there exists a vertex $v$ with $\widehat W\subseteq B_1(v,X)$.
\end{de}

\rem Note that the local conditions ($SD_2^{\ast}$) considered in the current paper are weaker than the ones in \cite{OCh}. First, our definition of a wheel with a pendant triangle does not assume that it is full. Then, the main difference is that we allow $4$--cycles to appear (but not $4$--wheels). This implies that our main result here, Theorem~\ref{logl} below, is stronger; cf.\ Remark after its proof.

\begin{lem}
\label{fulwheel}
Let $X$ satisfy the $SD_2^{\ast}$ property and let $\widehat W=(v_0;v_1,\ldots,v_k;t)$ be a $5$--wheel with a pendant triangle contained in $X$. If $t\nsim v_0$ then $\widehat W$ is a full subcomplex of $X$.
\end{lem}
\begin{proof}
Assume that $t\nsim v_0$. We have to prove that $t\nsim v_i$, for $i=3,4,5$.

Suppose that $t\sim v_3$. Then, since $t\nsim v_0$ and $v_1\nsim v_3$, we have that there exists a $4$--wheel $(v_2;t,v_3,v_0,v_1)$ --- contradiction. Thus $t\nsim v_3$ and, by an analogous argument, $t\nsim v_5$.

Suppose that $t\sim v_4$. Let $v$ be a vertex joinable with all vertices of $\widehat W$. Then, since $t\nsim v_0$ and $v_1\nsim v_3$, we have that
 there exists a $4$--wheel $(v;t,v_1,v_0,v_4)$ --- contradiction. Thus $t\nsim v_4$
 and the lemma is proved.
\end{proof}

\begin{lem}[$SD_2\Rightarrow SD_2^{\ast}$]
\label{sd2>sd2*}
If $X$ satisfies the condition $SD_2(v)$, for every vertex $v\in X$, then it satisfies also the condition $SD_2^{\ast}$.
\end{lem}
\begin{proof}
It is clear that $X$ as in the lemma satisfies the condition (a) of Definition \ref{sd2*}. Thus we show only that $X$ satisfies condition (b).
Let $\widehat W=(v_0;v_1,v_2,v_3,v_4,v_5;t)$ be a $5$--wheel with a pendant
triangle.

Observe that  $d(v_4,t)\leqslant 2$. If this was not true then
$d(v_4,t)=3$ and, by the property $SD_2(t)$ (Definition \ref{sdn(A)new}), we would have that
$v_5\sim v_3$ --- contradiction with the definition of a $5$--wheel.

If $d(v_4,t)=1$ then, by property $SD_2(t)$, we have that either
$v_0\sim t$ or $v_1\sim v_4$. In the first case we have $\widehat W\subseteq B_1(v_0)$ and we are done. The second case contradicts the definition of $\widehat W$.

Thus further we assume that $d(v_4,t)=2$ and that $v_0\nsim t$.
By the property $SD_2(t)$, there exists a vertex
$v\in S_1(t)$
such that $\langle v_0,v_4,v\rangle\in X$. Consider the cycle $(t,v_1,v_0,v,t)$.
By property $SD_2(t)$ we have that $v_1\sim v$.
Similarly, we have $v_2\sim v$. Analogous considerations for cycles $(v,v_2,v_3,v_4,v)$ and $(v,v_1,v_5,v_4,v)$ yield that $v \sim v_3,v_5$. This proves the lemma.
\end{proof}

\rem For $k=4,5$, the $k$--cycle satisfies condition $SD_2^{\ast}$ and does not satisfy condition $SD_2(v)$ for any vertex $v$.

\begin{tw}[Cartan-Hadamard theorem]
\label{logl}
Let $X$ be a simplicial complex satisfying the condition $SD_2^{\ast}$. Then its universal
cover is weakly systolic and in particular contractible.
\end{tw}
\begin{proof}
We construct the universal cover $\widetilde X$ of $X$ as an increasing
union $\bigcup_{i=1}^{\infty}\widetilde B_i$ of combinatorial balls. The covering map is then the union
$$
\bigcup_{i=1}^{\infty}f_i\colon \bigcup_{i=1}^{\infty}\widetilde B_i \to X,
$$
where
$f_i \colon \widetilde B_i \to X$ is locally injective and $f_i |_{\widetilde B_j}=f_j$, for $j\leqslant i$.

We proceed by induction. Choose a vertex $v$ of $X$. Define $\wt B_0=\lk v \rk$, $\widetilde B_1=B_1(v,X)$ and $f_1=\mr {Id}_{B_1(v)}$.
Assume that we have constructed the balls $\widetilde B_1,\widetilde B_2,\ldots,\widetilde B_i$ and
the corresponding maps $f_1,f_2,\ldots,f_i$ to $X$ so that the following conditions are satisfied:
\medskip

\noindent
\begin{description}
  \item[($P_i$)] $\wt B_j=B_j(v,\wt B_i)$ for $j=1,2,\ldots,i$;
  \item[($Q_i$)] $\wt B_i$ satisfies the $SD_{i-1}(v)$ property;
  \item[($R_i$)] $f_i|_{B_1(\wt w,\wt B_i)}\colon B_1(\wt w,\wt B_i) \to B_1(f_i(\wt w),X)$ is
injective for $\wt w\in \wt B_{i}$ and it is
an isomorphism for $\wt w\in \wt {B}_{i-1}$.
\end{description}

Observe that those conditions are satisfied for $\wt B_1$ and $f_1$, i.e.\ that conditions ($P_1$), ($Q_1$) and ($R_1$) hold.
Now we construct $\widetilde{B}_{i+1}$ and the map
$f_{i+1}\colon \widetilde{B}_{i+1}\to X$. For a simplex $\wt {\sigma}$ of
$\wt B_i$, we denote by $\sigma$ its image $f_i(\wt {\sigma})$ in $X$.
Let $\wt S_i=S_i(v,\wt B_i)$ and let
\begin{align*}
Z=\lk (\widetilde w,z)\in \widetilde S_i^{(0)}\times X^{(0)}|\; \; z\in X_{w}\setminus
f_i((\widetilde B_i)_{\widetilde{w}}) \rk.
\end{align*}
Define a relation $\stackrel{e}{\sim}$ on $Z$ as follows:
\begin{align*}
(\widetilde{w},z)\stackrel{e}{\sim} (\widetilde w',z') \; \; \mr{iff} \; \; (z=z' \;\; and \;\; \la \widetilde{w},\widetilde w'\ra \in \wt B_i^{(1)}).
\end{align*}

\medskip

\noindent
{\bf Claim 1.} The relation $\stackrel{e}{\sim}$ is an equivalence relation.
\medskip

\noindent
{\bf Proof of Claim 1.} To show the claim it is enough to show
that $\stackrel{e}{\sim}$ is transitive. Let $(\widetilde{w},z)\stackrel{e}{\sim}(\widetilde w',z')$ and
$(\widetilde w',z')\stackrel{e}{\sim} (\widetilde w'',z'')$. We will prove that then
$(\widetilde{w},z)\stackrel{e}{\sim} (\widetilde w'',z'')$.
This will follow easily from the following claim.
\medskip

\noindent
{\bf Claim 2.} The vertices $w$ and $w''$ coincide or are connected by an edge in $X$.
\medskip

\noindent
{\bf Proof of Claim 2.}
By conditions ($P_i$) and ($Q_i$) we have the following. There exist simplices
$\widetilde{\rho},\widetilde{\sigma},\widetilde{\tau}\in \widetilde B_{i-1}$ such that
$\langle\widetilde{w} \cup \widetilde w' \cup \widetilde{\rho}\rangle,\langle\widetilde w'\cup \widetilde{\sigma}\rangle,
\langle\widetilde w' \cup \widetilde w'' \cup \widetilde{\tau}\rangle$ are simplices in $\wt B_i$. Observe that
$\widetilde{\rho}, \widetilde{\tau} \subseteq \widetilde{\sigma}$.

Consider the case $i=1$ for which $\sigma=\rho=\tau=v$. If $d(z,v)\leqslant 1$ then $z\in f_i((\wt B_1)_{\wt w})$ and thus
$(\wt w,z)\notin Z$ --- contradiction. Thus $d(z,v)=2$. If $d(w,w'')=2$ then
there exists a $4$--wheel $(w';v,w,z,w'')$ in $X$. This contradicts the condition $SD_2^{\ast}$ for $X$. Hence $d(w,w'')\leqslant 1$ and Claim 2 holds
in the case $i=1$.

Thus for the rest of the proof of Claim 2 we assume that $i\geqslant 2$.
Then there exists a vertex $\widetilde t\in \wt B_{i-2}$ joinable with $\wt {\sigma}$.
Assume that $d(w,w'')=2$ --- this will lead us to contradiction.

First observe that $d(z,u)=2$ for every vertex $u\in \sigma$. It is because if $d(z,u)\leqslant 1$ then $z\in f_i((\wt B_i)_{\wt w'})$ and hence $(\wt w',z) \notin Z$ --- a contradiction.
Moreover, by condition ($R_i$), we have $d(w,t)=d(w',t)=d(w'',t)=2$.

Choose two vertices: $u\in \rho$ and $u''\in \tau$.
We show that $d(w,u'')=2$. It is clear that $w\neq u''$ and if $d(w,u'')=1$ then there exists a $4$--wheel $(w';u'',w,z,w'')$ --- contradiction with the $SD_2^{\ast}$ property.
Similarly, we have $d(w'',u)=2$.
Thus there is a $5$--wheel with a pendant triangle $\widehat {W}_5=(w';u,u'',w'',z,w;t)$ in $X$.
By the condition $SD_2^{\ast}$ there exists in $X$ a vertex $w_c\neq w'$ such that
$\widehat {W}_5 \subseteq X_{w_c}$. By the condition $(R_i)$ we have that there exist a vertex $\wt u \in \wt{\tau}$ and a vertex $\wt{w_c}\in (\wt {B_i})_{\wt u}$ such that $f_i(\wt{w_c})=w_c$, $f_i(\wt u)=u$.

Observe that, by the condition ($R_i$) (for $B_1(\wt u,\wt B_i)$), we have
that $d(\wt w_c,\wt {w})= 1$ and $d(\wt w_c,{\wt t})= 1$.
Thus $d(\wt {w_c},\wt v)=i-1$.
Since $\langle w_c,z\rangle\in X$ we have that $z\in f_i((\wt{B_i})_{\wt{w}})$ and hence
$(\wt{w},z)\notin Z$ --- a contradiction. This shows that $d(w,w'')\neq 2$ and finishes the proof of Claim 2.
\medskip

We come back to the proof of Claim 1. Choose again $u\in \rho$ and $u''\in \tau$. If $w\nsim u''$ and $w''\nsim u$ (both in $X$) then we have a $4$--wheel
$(w';w,w'',u'',u)$ --- contradiction. W.l.o.g. we can thus assume that $w''\sim u$. Then, by condition ($R_i$) (for $B_1(\wt u ,\wt {B_i})$), we have that
$\langle\wt w,\wt {w''}\rangle\in \wt {B_i}$ and
$(\widetilde{w},z)\stackrel{e}{\sim} (\widetilde{w''},z'')$. This finishes the proof of Claim 1.
\medskip

We define the flag simplicial complex $\wt{B}_{i+1}$ as follows.
Its $0$--skeleton is by definition the set
$\wt{B}_{i+1}^{(0)}=B_{i}^{(0)}\cup (Z/\stackrel{e}{\sim})$.
Now we define the $1$--skeleton $\wt{B}_{i+1}^{(1)}$ of $\wt{B}_{i+1}$ as follows.
Edges between vertices of $\wt {B_i}$ are the same as in $\wt {B_i}$.
Moreover, for every $\wt w\in \wt{S_i}^{(0)}$, there are edges joining $\wt {w}$ with $[\wt {w},z] \in Z/\stackrel{e}{\sim}$ (here
$[\wt {w},z]$ denotes the equivalence class of $(\wt {w},z)\in Z$) and
the edges joining $[\wt {w},z]$ with $[\wt {w},z']$, for $\langle z,z'\rangle\in X$.
Having defined $\wt{B}_{i+1}^{(1)}$ the higher dimensional skeleta are
determined by the flagness property.
\medskip

Definition of the map $f_{i+1}\colon \wt{B}_{i+1}^{(0)}\to X$ is clear:
$f_{i+1}|_{\wt {B_i}}=f_i$ and $f_{i+1}([\wt {w},z])=z$.
We show that it can be simplicially extended.
It is enough to do it for simplices in $\wt{B}_{i+1}\setminus \wt {B_{i-1}}$.
Let
$\wt {\sigma}=\langle[\wt {w_{1}},z_1],\ldots,[\wt {w_{l}},z_l],\wt {w_1'},\ldots,\wt {w_m'}\rangle\in \wt{B}_{i+1}$ be a simplex.
Then, by definition of edges in $\wt{B}_{i+1}$, we have that
$\langle z_p,z_q\rangle\in X$ and $\langle z_r,w_s'\rangle\in X$, for $p,q,r\in \lk 1,2,\ldots,l\rk$
and $s\in \lk 1,2,\ldots,m \rk$.
Since $f_{i+1}([\wt {w_p},z_p])=z_p$, $f_{i+1}(\wt {w_s'})=w_s'$ and
since $f_i$ was simplicial, it follows that
$\lan \lk f_{i+1}(\wt {w})|\;\; \wt {w}\in \wt {\sigma}\rk \ran \in X$. Hence, by the simplicial extension, we can define the map $f_{i+1}\colon \wt{B}_{i+1}\to X$.
\medskip

Now we check that $\wt{B}_{i+1}$ and $f_{i+1}$ satisfy conditions ($P_{i+1}$), ($Q_{i+1}$) and ($R_{i+1}$).
\medskip

\noindent
{\bf Condition} ($P_{i+1}$). Since for every $[\wt {w},z]\in \wt{B}_{i+1}$ we have $d(v, [\wt {w},z])=i+1$, it is clear that $\wt {B_j}=B_j(v, \wt{B}_{i+1})$, for $j=3,4,\ldots,i+1$. Thus ($P_{i+1}$) holds.
\medskip
\medskip

\noindent
{\bf Condition} ($R_{i+1}$).
Observe that, since the property ($R_i$) is satisfied by $B_i$, it is enough to consider only vertices $\wt w\in \wt{B}_{i+1}\setminus \wt {B_{i-1}}$. We consider separately the two following cases.
\medskip

\noindent
\emph{(Case 1: $\wt w\in \wt {S_i}$.)}
First we prove injectivity of the map $f_{i+1}|_{B_1(\wt w, \wt{B}_{i+1})}$. Let $\wt x\neq \wt {x'}\in (\wt{B}_{i+1})_{\wt w}$ be two vertices.
If $\wt x,\wt {x'}\in \wt {S_i}$, then $f_{i+1}(\wt x)\neq
f_{i+1}(\wt {x'})\neq f_{i+1}(\wt w)\neq f_{i+1}(\wt {x})$, by local injectivity of $f_i$.
If $\wt x\in \wt {S_i}$ and $\wt {x'}=[\wt w,z]$ then
$f_{i+1}(\wt x)\neq
f_{i+1}(\wt {x'})=z\neq f_{i+1}(\wt w)\neq f_{i+1}(\wt {x})$, by local injectivity of $f_i$ and by the definition of the set $Z$ containing $z$.
If $\wt {x}=[\wt w,z]$ and $\wt {x'}=[\wt w,z']$ then
$f_{i+1}(\wt x)=z\neq
f_{i+1}(\wt {x'})=z'\neq f_{i+1}(\wt w)\neq f_{i+1}(\wt {x})=z$ by the definition of $Z$ and by the fact that $z\neq z'$.

Surjectivity of the map $f_{i+1}|_{B_1(\wt w, \wt{B}_{i+1})}$ follows from the fact that, for $z\in X_w\setminus f_i(B_1(\wt w, \wt{B_i}))$, we have $f_{i+1}([\wt w, z])=z$.
\medskip

\noindent
\emph{(Case 2: $\wt w=[\wt y,z]$.)}
Here we prove only injectivity of $f_{i+1}|_{B_1(\wt w, \wt{B}_{i+1})}$.
Let $\wt x\neq \wt {x'}\in (\wt{B}_{i+1})_{\wt w}$ be two vertices.
If $\wt x,\wt {x'}\in \wt {S_i}$, then $f_{i+1}(\wt x)\neq
f_{i+1}(\wt {x'})\neq f_{i+1}(\wt w)=z\neq f_{i+1}(\wt {x})$, by local injectivity of $f_i$ and by the definition of the set $Z$.
If $\wt x\in \wt {S_i}$ and $\wt {x'}=[\wt {y'},z']$ then
$f_{i+1}(\wt x)\neq
f_{i+1}(\wt {x'})=z'\neq f_{i+1}(\wt w)=z\neq f_{i+1}(\wt {x})$, by the definition of the set $Z$ and since $z\neq z'$.
If $\wt {x}=[\wt {y'},z']$ and $\wt {x'}=[\wt {y''},z'']$ then
$d(\wt {y'},\wt{y''})\leqslant 1$ and thus $z'\neq z''$. Since $z'\neq z\neq z''$ we have
$f_{i+1}(\wt x)=z'\neq
f_{i+1}(\wt {x'})=z''\neq f_{i+1}(\wt w)=z\neq f_{i+1}(\wt {x})=z'$,
\medskip

\noindent
{\bf Condition} ($Q_{i+1}$).
By the condition ($Q_i$) it is enough to show the following.
Let $\sigma$ be a simplex in $\wt {S_{i+1}}$. Then $(\wt{B}_{i+1})_{\sigma}\cap \wt{B_i}$ is a non-empty simplex (cf.\ Definition \ref{sdn(A)new}).

By the definition of edges in $\wt {S_{i+1}}$ it is clear that
$(\wt{B}_{i+1})_{\sigma}\cap \wt{B_i}$ is non-empty.
Thus, by Lemma \ref{sd vert}, it is now enough to show that for every vertex $[\wt {w},z]$, if
$\langle[\wt {w},z],\wt{w'}\rangle,\langle[\wt {w},z],\wt {w''}\rangle\in \wt{B}_{i+1}$ then
$\langle\wt {w'},\wt {w''}\rangle\in \wt{B}_{i+1}$.
If $\wt{w'},\wt{w''}$ are as above then, by definition of edges
in $\wt{B}_{i+1}$, we have $[\wt {w},z]=[\wt {w'},z]$ and
$[\wt {w},z]=[\wt {w''},z]$. Thus, by definition of $\stackrel{e}{\sim}$ we get
that $\langle\wt {w'},\wt {w''}\rangle\in \wt {B_i}$. Hence ($Q_{i+1}$) follows.
\medskip

Having established conditions ($P_{i+1}$), ($Q_{i+1}$) and ($R_{i+1}$) we conclude that, inductively, we construct a complex
$\wt X=\bigcup_{i=1}^{\infty} \wt {B_i}$ and a map
$f=\bigcup_{i=1}^{\infty} f_i\colon \wt X \to X$ with the following properties.
The complex $\wt X$ satisfies the property $SD_n(v)$ for every $n$ and the map $f$ is a covering map. Thus, by Proposition \ref{sdncontr}, the complex $\wt X$ is contractible and in particular it is the universal cover of $X$.
Since the vertex $v$ was chosen arbitrarily in our construction and since the universal cover of $X$ is unique it follows that $\wt X$ satisfies the property $SD_n(v)$ for every vertex $v$ and for every natural number $n$.
This finishes the proof of the theorem.
\end{proof}

\rem The above proof of the local-to-global result is a combinatorial version of the well known construction of the universal cover, as a quotient of a space of paths. Similar proof was first given in the systolic case in \cite[Section 4]{JS1}. In \cite{OCh} we present another proof of the Cartan-Hadamard theorem for weakly systolic complexes, basing on the minimal diagram technique. However, the assumptions there are stronger (cf.\ Remark after Definition~\ref{sd2*} above), hence the corresponding result is weaker than Theorem~\ref{logl} above. Moreover, the scheme for proving local-to-global statements presented in the proof above is useful in many other situations (where other methods seem not to work). We use it (nearly following the notations from here) in \cite{BCC+} for a class of complexes generalizing weakly systolic and CAT(0) cubical complexes (note however that the result there is again weaker if restricted to weakly systolic case). More intriguingly, the same method is used for a construction of the universal cover in a ``positive curvature" case in \cite{CCO}.

\section{Convexity}
\label{convex}

Convexity of some balls is a crucial property of weakly systolic complexes, implying many nonpositive-curvature-like properties. In this section we show few, rather technical results concerning convexity. They are used e.g.\ in Section \ref{neg}.

\begin{de}[Convexity]
\label{3convex}
A subcomplex $Z$ in a flag simplicial complex $X$ is
\emph{convex} if it is full, connected, and if for every two vertices $v,w\in Z$, each $1$--skeleton geodesic between $v$ and $w$ (in $X$) is contained in $Z$. The subcomplex
$Z$ is \emph{$3$--convex} if it
is full, and for any $1$--skeleton geodesic $(v_1, v_2, v_3)$ in $X$, if $v_1, v_3\in Z$ then $v_2\in Z$. A subcomplex $Z$ is \emph{locally $3$--convex} if it is full, and if for every vertex $v\in Z$ the link $Z_v$ is $3$--convex in the corresponding link $X_v$.
\end{de}

\rem The notion of $3$--convexity appears e.g.\ in \cite[Section 4]{Ch-class} (under the name ``\emph{local convexity}"), and in {\cite[Section 3]{JS1}}, where the term ``$3$--convexity" is introduced.

\begin{lem}[loc.\ convex $\equiv$ convex]
\label{lco>co}
For $Z$ being a subcomplex of a weakly systolic complex $X$, the following conditions are equivalent.

\begin{enumerate}
\item
$Z$ is convex.

\item
$Z$ is connected and locally $3$--convex.

\item
$Z$ is connected and $3$--convex.
\end{enumerate}
\end{lem}
\dow
The implications (1)$\Rightarrow$(2) and (1)$\Rightarrow$(3) are obvious.
\medskip

\noindent
(2)$\Rightarrow$(3).
Let $(z,v,w)$ be a $1$--skeleton geodesic in $X$ with $w,z\in Z$. By connectedness, there is a $1$--skeleton path $\gamma=(v_0=w,v_1,\ldots,v_l=z)$ being a full subcomplex of $Z$. Assume that $\gamma$ is chosen so that the sum $S(\gamma)=\sum_{i=0}^{l} d(v,v_i)$
is minimal (among full paths connecting $z$ and $w$ in $Z$).
If $l=2$ then, by the property $SD_1(w)$ we have that $v_1\sim v$ and thus, by the local $3$--convexity (at vertex $v_1$) $v\in Z$.

Thus further we assume that $l\geqslant 3$. If $v_i=v$ for some $i$ then we are done. Assume this is not the case. If $\gamma \subseteq X_v$ then, by the local $3$--convexity at $v_1$ we get $v\in Z$. Hence, it remains to consider the case when $d(v_i,v)\geqslant 2$, for some $i$. Let $i$ be the smallest number such that $n=d(v,v_i)$ is maximal. Then $d(v_{i-1},v)=n-1$ and, by the $SD_{n-1}(v)$ property, there exists a vertex $v_{i}'\in S_{n-1}(v,X)$ such that $v_i'\sim v_{i-1},v_{i+1}$. By the local $3$--convexity of $Z$ (at $v_i\in Z$) we have that $v_i'\in Z$.
Consider now the path $\gamma'=(v_0,v_1,v_{i-1},v_i',v_{i+1},\ldots,v_l)$. One can find a full path $\gamma''$ containing some (possibly all) vertices of $\gamma'$, and no other vertices.
However, then we have
\begin{align*}
&S(\gamma'')\leqslant S(\gamma')=\\
&=d(v,v_0)+\cdots + d(v,v_{i-1})+d(v,v_i')+d(v,v_{i+1})+\cdots +d(v,v_l) = S-1.
\end{align*}
This contradicts the choice of $\gamma$.
\medskip

\noindent
(3)$\Rightarrow$(1).
 Connectedness together with the $3$--convexity imply convexity by \cite[Theorem 7]{Ch-class} (weakly bridged graphs are obviously \emph{weakly modular}; cf.\ \cite{OCh}).
\kon

\begin{cor}
\label{ballvert}
Balls around vertices in weakly systolic complexes are convex.
\end{cor}
\dow
Clearly, balls around vertices are connected. Vertex condition (V) from Definition~\ref{sdn(A)} implies immediately their $3$--convexity.
\kon

\begin{lem}
\label{conv>contr}
A convex subcomplex of a weakly systolic complex is weakly systolic.
\end{lem}
\dow
This follows easily from the definition of convexity and from Definition \ref{sdn(A)} of the property $\wt {SD}_n(A)$.
\kon

\rem In general balls around simplices in weakly systolic complexes are not convex. As an example consider the ball of radius one around an edge in the $5$--wheel, not adjacent to the central vertex. The following results establish convexity of some particular balls in weakly systolic complexes, and the corresponding simple descent properties.

\begin{lem}[Edges descend on balls]
\label{edge_desc}
Let $\sigma$ be a simplex of a weakly systolic complex $X$.
Let $e=zz'$ be an edge contained in the sphere $S_i(\sigma).$
Then there exists a vertex $w\in \sigma$ and a vertex $v\in B_{i-1}(\sigma)$ such that $v$
is adjacent to $z,z'$ and $d(v,w)=i-1.$
\end{lem}

\begin{proof}
If there exists a vertex $w\in \sigma$ such that $z,z'\in S_i(w)$ then
the assertion follows from the edge condition (E) of Definition~\ref{sdn(A)} of the property $\wt{SD}_i(w)$. Thus further we assume that such a vertex of $X$ does not exists. Let $w,w'$  be two
vertices of $\sigma$ with $d(w,z)=d(w',z')=i.$  Since $d(w',z)=d(w,z')=i+1,$ we
conclude that $z$ belongs to a geodesic connecting $w$ and $z'$. Since $w,z'\in B_i(w')$ and $z\notin B_i(w'),$ this contradicts
the convexity of $B_i(w').$
\end{proof}

\begin{lem}[Big balls are convex]
\label{bbac}
Let $\sigma$ be a simplex of a weakly systolic complex $X$ and let $i\geqslant 2.$ Then
the ball $B_i(\sigma)$ is convex. In particular, $B_i(\sigma)\cap X_z$ is a simplex for any vertex
$z\in S_{i+1}(\sigma)$.
\end{lem}

\begin{proof}
To prove convexity it is enough, by Lemma~\ref{lco>co}, to prove local $3$--convexity. Let a vertex $z$ be adjacent to $x,y\in B_i(\sigma)$, with $d(x,y)=2$. We have to show that $z\in B_i(\sigma)$. Suppose by way of contradiction that $z\in S_{i+1}(\sigma)$.
Let $u$ and $v$ be  vertices
of $\sigma$ located at distance $i$ from $x$ and $y$, respectively. If $u=v$ then, by the property $\wt{SD_i}(u)$, the vertices $x,y$ must be adjacent. So, suppose that $u\neq v$, and $d(y,u)=d(z,u)=i+1$. By the edge condition (E) (of Definition~\ref{sdn(A)}), there exists a common neighbor $w$ of
$z$ and $y$, at distance $i$ from $u$.
Then, by vertex condition (V) (applied to $z$ and $u$), the vertices $x$ and $w$ are adjacent.
Again by edge condition (E), there exists a common neighbor $u'$ of $w$ and $x$ at distance $i-1$ from $u$.  If $d(w,v)=i+1$
then $y$ and $u'$ must be adjacent, by vertex condition (V) (for vertices $w$ and $v$). As a result, we obtain a $4$--cycle
defined by $x,z,y,u'$. Since $d(z,u)=i+1$ and $d(u',u)=i-1$, vertices $z$ and $u'$ cannot be adjacent, thus vertices $w,x,y,z,u'$ span a $4$--wheel, which is impossible.
Hence $d(w,v)=i$. Let $u''$ be a neighbor of $u$ on the geodesic between $u'$ and $u$ (it is possible that $u''=u'$). Since $d(y,u)=i+1$
and $d(u',u)=i-1, d(u',y)=2$, we conclude that $u'$ belongs to a geodesic between $y$ and $u$, implying that $u''$ belongs to a geodesic between $u$ and $y$. Since $v$ also belongs
to the latter geodesic, by vertex condition (V), the vertices $u''$ and $v$ must be adjacent. But in this case $d(x,v)=1+d(u',u'')+1=i$,
contrary to the assumption that $d(x,v)=i+1$.  This contradiction shows that $B_i(\sigma)$ is convex for any $i\geqslant 2.$
\end{proof}

\begin{prop}[$SD_n$ property for maximal simplices]
\label{SD_n-max}
A weakly systolic complex satisfies the property $SD_n(\tau)$ for any maximal simplex $\tau$ and every $n$.
\end{prop}

\begin{proof} Let $\sigma$ be a simplex of the sphere $S_{i+1}(\tau)$.  For each vertex $v\in \sigma$, denote by
$\tau(v)$ the metric projection of $v$ in $\tau,$ i.e., the set of all vertices of $\tau$ located at distance $i+1$ from $v.$
Notice that the sets $\tau(v)$ $(v\in \sigma)$ can be linearly
ordered by inclusion. Indeed, if we suppose the contrary, then there exist two vertices $v',v''\in \sigma$, and
the vertices $u'\in \tau(v')\setminus \tau(v'')$ and
$u''\in \tau(v'')\setminus \tau(v')$. This would however contradict the convexity
of $B_{i+1}(u')$. It follows that  $\sigma\subset S_{i+1}(u)$ holds
for any vertex $u$ belonging to all metric projections   $\tau_0=\cap\{ \tau(v)\; | \; v\in \sigma\}$. Applying the $SD_n(u)$ property to $\sigma$ we conclude that the set of all vertices $x\in S_i(u)\subseteq S_i(\tau)$
adjacent to all vertices of $\sigma$ is a non-empty simplex. Pick two vertices $x,y\in S_i(\tau)$ adjacent to all vertices of $\sigma$.  Let $x\in S_i(u)$ and $y\in S_i(w)$ for $u,w\in \tau_0$. We assert that $x$ and $y$ are  adjacent. Let $v$ be a vertex
of $\sigma$ whose projection $\tau(v)$ is maximal by inclusion. If $\tau(v)=\tau$
then applying the $SD_n(v)$ property we conclude that there exists a vertex $v'$ at distance $i$ to $v$ and adjacent to all vertices of $\tau$ contrary
to maximality of $\tau$. Hence $\tau(v)$ is a proper simplex of $\tau$. Let $s\in \tau\setminus \tau(v)$. Then $x,y$ belong to geodesics between $w$ and $s$, and by vertex condition (V), the vertices $x$ and $y$ must be adjacent.
\end{proof}

\begin{cor}
\label{convmax}
In weakly systolic complexes balls around maximal simplices are convex.
\end{cor}

\section{Examples of weakly systolic complexes and groups}
\label{ex}
In this section we provide several classes of examples of weakly systolic complexes and groups. Those are: systolic complexes and groups (Subsection \ref{syst}), ``CAT(-1) cubical" groups (Subsection \ref{cc-1}), lattices in isometry groups of right-angled hyperbolic buildings (Subsection \ref{build}) and some other (non-)examples (Subsection \ref{othex}).

\subsection{Systolic complexes}
\label{syst}

Here we show elementarily (cf.\ Remarks after Corollary \ref{systcontr}) that systolic complexes (cf.\ \cites{Ch-CAT, Ha, JS1}) are weakly systolic. In particular this gives a simple proof of the contractibility of systolic complexes.
\medskip

Recall, that for $k\geqslant 4$ a flag simplicial complex is \emph{$k$--large} if   every full cycle in $X$ has length at least $k$. A flag simplicial complex is \emph{locally $k$--large} if every its link is $k$--large. A (connected and) simply connected locally $k$--large simplicial complex is called \emph{$k$--systolic} and a $6$--systolic complex is called just \emph{systolic}.

Let $X$ be a simplicial complex and let $D$ be a triangulation of a $2$--disk.
Following \cite{Ch-CAT} we call a simplicial map $f\colon D\to X$ a \emph{disk diagram for the cycle $f(\partial D)$}. A disk diagram for a cycle $C\subseteq X$ is \emph{minimal} if $D$ has minimal number of $2$--simplices. A disk diagram is \emph{non-degenerate} if it is injective  on each simplex of $D$.
The proof of the following simple lemma can be found in \cite[Lemma 5.1 and proof of Theorem 8.1]{Ch-CAT} and \cite[Lemma 1.6]{JS1}.

\begin{lem}[Systolic filling]
\label{systfil}
Let $X$ be a systolic complex and let $C$ be a simple (i.e.\
without self-intersections) cycle in $X$. Then there exists a
non-degenerate disk diagram $f\colon D\to X$ for $C$ --- a minimal disk
diagram --- such that $f|_{\partial D}\colon \partial D\to C$ is an
isomorphism and every interior vertex of $D$ is contained in at least
$6$ triangles.
\end{lem}

The proof of the following lemma is essentially the same as the proof of Theorem 8.1 in \cite{Ch-CAT}. We provide it for completeness.

\begin{prop}
\label{syst>sdn}
Every systolic complex is weakly systolic.
\end{prop}
\begin{proof}
Let $X$ be a systolic complex.
Let $i$ be a natural number and let $v$ be a vertex.
First we check the vertex condition (V) of Definition \ref{sdn(A)}. Let
$w\in S_{i+1}(v, X)$ be a vertex.
Let $\gamma=(v_0=v,v_1,\ldots,v_{i+1}=w)$ and
$\gamma'=(v_0'=v,v_1',\ldots,v_{i+1}'=w)$
be two $1$--skeleton geodesics. We have to show that $\langle v_i,v_i'\rangle\in X$.
If there exists $0<j\leqslant i$ such that $v_j=v_j'$ then we can
restrict our analysis to the cycle $(v_j,\ldots,v_i,w,v_i',\ldots,v_j',v_j)$.
Thus w.l.o.g. we can assume that it is not the case, i.e.\ that the
cycle $C=(v_0,\ldots,v_i,w,v_i',\ldots,v_0',v_0)$ is simple.
Let $f\colon D\to X$ be a minimal disk diagram for $C$ as in Lemma
\ref{systfil}. Since the Euler characteristics $\chi(D)$ of the disk
$D$ equals $1$, a \emph{combinatorial Gauss-Bonnet Formula} (compare
e.g.\ \cite[Section 3]{Ch-CAT} or \cite[Section 1]{JS1}) gives the
following (we use here the same notation for the vertices in $D$ and for their images in $X$):
$$
1=\chi (D)=\frac{1}{6} \left( \sum_{z\in {\mr {int}}{D}}(6-\chi(z))+
\sum_{v\in \partial D}(3-\chi(z))
\right),
$$
where the first sum is taken over the vertices in the interior
$\mr{int}{D}$ of $D$, the second one over vertices on the boundary
$\partial D$ of $D$ and
$\chi(z)$ denotes the number of triangles in $D$ containing $z$.
Observe that since $\gamma,\gamma'$ are geodesics we have
$\chi(v_j),\chi(v_j')\geqslant 2$ for $j=1,\ldots,i$ and, moreover,
if $\chi(v_j)=\chi(v_k)=2$ (or $\chi(v_j')=\chi(v_k')=2$), for some
$j<k$, then there exists $j<l<k$ with
$\chi(v_l),\chi(v_l)>3$ ($\chi(v_l'),\chi(v_l')>3$).
Thus $\sum_{j=1}^{i} (3-\chi(v_j))+ \; \sum_{j=1}^{i}
(3-\chi(v_j'))\leqslant 2$. Since $\chi(z)\leqslant 0$ for $z\in \mr{int}D$, we
have by Gauss-Bonnet formula that $\chi(v)+\chi(w)\leqslant 2$. Thus we
get that $\chi(v)=\chi(w)=1$ that implies $v_i\sim v_i'$ and
finishes the proof of the vertex condition.

Now we go to the edge condition (E) of Definition \ref{sdn(A)}. Let
$e=\langle w,w'\rangle\in S_{i+1}(v, X)$ be an edge. Choose $1$--skeleton geodesics
$(v_0=v,v_1,\ldots,v_{i+1}=w)$, $(v_0'=v,v_1',\ldots,v_{i+1}'=w')$. As before we can assume that $C=(v_0,\ldots,v_{i+1},v_{i+1}',\ldots,v_0',v_0)$ is a simple cycle and we can consider a minimal disk diagram $f\colon D\to X$ for $C$ as in Lemma \ref{systfil}. Again, using the combinatorial Gauss-Bonnet formula, we get that $\chi (v)+\chi(w)+\chi(w')\leqslant 5$. If $\chi(w)=1$ (or $\chi(w')=1$) then $v_i\sim w'$ ($v_i'\sim w$) thus
$\la v_i,e\ra \in X$ ($\la v_i',e \ra \in X$) and the edge condition is proved. Assume that $\chi(w),\chi(w')\rangle1$. Then we have to have $\chi(w)=\chi(w')=2$.
Let $\ov u\in D$ be the vertex which spans a simplex with $\ov e$.
Since $u\sim v_i$ we have that $i\leqslant d(u,v)\leqslant i+1$. If we prove that $d(u,v)=i$ then we are done. Assume that this is not the case, i.e.\ that $d(u,v)=i+1$. Then, by the vertex condition proved above we have that $v_i\sim v_i'$ since $\la v_i, u \ra,\la v_i',u\ra  \in D$.
The cycle $(v_i, w,  {w'},  v_i',v_i)$ has
a diagonal in $D$ and we get contradiction with $\chi(w)=1$ or with $\chi(w')=2$.
\end{proof}

By applying Proposition \ref{sdncontr} we get the following.

\begin{cor}
\label{systcontr}
A finite dimensional systolic complex is contractible.
\end{cor}

\rem The first implicit proof of the contractibility of systolic complexes can be found in \cite{Ch-CAT} where it is proved (see \cite[Theorem 8.1]{Ch-CAT}) that systolic complexes are bridged. By \cite{AnFa} (see also \cite{Ch-bridged}) bridged complexes are dismantlable --- this property is stronger than contractibility. The explicit proof of Corollary \ref{systcontr} appears in \cite[Theorem 4.1.1]{JS1}. There it is proved that systolic complexes are weakly systolic by constructing directly a universal cover of a locally $6$--large complex, similarly as in our proof of Theorem \ref{logl}.
We decided to present in the current paper a self-contained proof of contractibility of systolic complexes (including proofs of Proposition \ref{sdncontr} and Proposition \ref{syst>sdn}) because it seems to be much simpler than the two approaches mentioned above and because it emphasizes the role of the $SD_n(v)$ property in the systolic setting.
\medskip

Yet another approach to proving Proposition \ref{syst>sdn} is to use Theorem \ref{logl} and the following lemma, whose proof is a direct consequence of the definitions and will be omitted here.

\begin{lem}
\label{systsd2*}
A locally $6$--large flag simplicial complex satisfies the property $SD_2^{\ast}$.
\end{lem}

\rem The class of systolic groups includes in particular: $\mathbb Z$, $\mathbb Z^2$, free non-abelian groups $\mathbb F_n$, $\mathbb F_2 \times \mathbb F_2$, fundamental groups of surfaces, some classes of small cancelation groups, systolic groups of arbitrarily large virtual cohomological dimension constructed in \cites{JS1,Ha,O-chcg} (look there for details). However, many classical groups are not systolic. In \cites{JS2,O-ciscg,O-ib7scg,OS} serious restrictions on systolic groups are studied and many non-examples are listed.

\subsection{CAT(-1) cubical groups}
\label{cc-1}
In this section we present the second most important class of weakly systolic groups: groups acting geometrically
on CAT(-1) cubical complexes; cf.\ Corollary \ref{cat-1cc}.
CAT(-1) cubical complexes provide examples of weakly systolic groups that are not systolic; cf.\ Remarks after Corollary \ref{cat-1cc}.
Actually, in Proposition \ref{5lnotr} we prove that some more general class of groups consists of weakly systolic groups.
\medskip

\rem We think that it would be very instructive for the reader to consider first the easiest cases of the following crucial Lemma \ref{k-l.thick}. For $k=4,5,6$ a version of that lemma is proved more elementarily in \cite{O-chcg}.

\begin{de}[Thickening]
\label{thick}
Let $Y$ be a simple cell complex.
The \emph{thickening} $Th(Y)$ of $Y$ is the simplicial complex
whose vertices are vertices of $Y$ and whose simplices correspond to sets of vertices of $Y$ contained in a common face.
\end{de}

\rem
The technique of the thickening was invented by T. Januszkiewicz and, independently, by the author; compare \cite[Section 8]{Sw-propi}.
In the case of cubical complexes, a construction similar to the thickening has been used in graph theory. For a graph $G$ being the $1$--skeleton of a cubical complex $Y$ (such graphs are called \emph{median graphs}), a graph $G^{\Delta}$ is defined as the $1$--skeleton of $Th(Y)$; cf.\ \cite{BaCh} (we use their notation here).

\begin{lem}
\label{combinat}
Let $k\geqslant 4$ be a natural number and let $A^i$
be a finite set, for $i=0,2,\ldots,k-1$. Let $\Gamma$ be a graph with the vertex set $\bigcup A^i$ and with the following properties:
\begin{enumerate}
\item
$\langle v,v'\rangle \in \Gamma$,
for every $i$ and every $v\in A^i$, $v'\in A^{i+1 (\mr {mod}\; k)}$,

\item
for every $i,i'$ with $i-i'\neq \pm 1 (\mr {mod}\; k)$, there exist
vertices $v^i_{i'}\in A^i$ and $v^{i'}_{i}\in A^{i'}$ not connected by an edge.

\end{enumerate}
Then, for some $l\leqslant k$, there exists a $1$--skeleton cycle $(v_1, v_2,\ldots,v_l,v_1)\subset \Gamma$ without a diagonal in $\Gamma$.
\end{lem}

\dow
We prove this by induction on $k$.

\noindent
({\bf Case $k=4$.}) The required cycle is $(v^0_{2},v^1_{3},v^2_{0}, v^3_1, v^0_2)$.
\medskip

\noindent
({\bf    Induction step $k \to k+1$.})
We assume that we proved the lemma for $k\geqslant 4$. Now we prove the lemma for $k+1$.

If there exist (in $\Gamma$) a cycle of length at most $k$ then we are done.
So for the rest of the proof we assume there is no such a cycle.
\medskip

\noindent
{\bf Claim 1.} For every $i=0,1,\ldots,k$ and $i'$ such that $i'-i=2,\ldots,k-1$ (mod $(k+1)$), we have
$v^i_{i+2}\nsim  v^{i'}_{i'-2}$ (addition mod $(k+1)$).
\medskip

\noindent
{\bf Proof of Claim 1.} We show the claim by induction on $m=i'-i$ (here and later we add mod $k$).
If $m=2$ then the claim follows from the assumptions on $\Gamma$.

Assume we proved the claim for $1,2,\ldots,m$. Now we show it for $i'-i=m+1$. Consider a subgraph $\Gamma '$ of $\Gamma$ spanned (induced)
by vertices in the set $\bigcup_{j=i+1}^{i'-1}A^i\cup \lk v^i_{i+2}, v^{i'}_{i'-2} \rk$.

By the induction assumptions we have that
$v^i_{i+2}\nsim  v^{j}_{j-2}$, for every $j=i+2,\ldots,i'-1$. And, analogously,
we have that $v^{i'}_{i'-2}\nsim  v^{j}_{j+2}$, for every $j=i+1,\ldots,i'-2$.
We want to show that $v^i_{i+2}\nsim  v^{i'}_{i'-2}$.
Assume it is not so, i.e.\ $v^i_{i+2}\sim v^{i'}_{i'-2}$.
Let $\ov A^0=\lk v^i_{i+2}\rk$, $\ov A^{m}=\lk v^{i'}_{i'-2} \rk$ and let $\ov A^j=A^{i+j}$, for every $j=1,2,\ldots,m-1$.
Then the family $\lk \ov A^j \rk$ and the graph $\Gamma'$ satisfy the hypotheses of the lemma. Thus by the induction (on $k$ --- observe that $m+2=i'-i+1\leqslant k$) assumptions there exists a cycle $\gamma$
in $\Gamma'$ of length at most $k$ and without a diagonal.
Then $\gamma$ is a cycle in $\Gamma$ with the same properties --- contradiction.
Hence $v^i_{i+2}\nsim  v^{i'}_{i'-2}$ and Claim 1 is proved.
\medskip

\noindent
{\bf Claim 2.} For every $i=0,1,\ldots,k$ and $i'$ such that $i'-i=2,\ldots,k-1$ (mod $(k+1)$), we have
$v^i_{i+2}\nsim  v^{i'}_{i'+2}$ (addition mod $(k+1)$).
\medskip

\noindent
{\bf Proof of Claim 2.}
We argue by contradiction. Assume $v^i_{i+2}\sim v^{i'}_{i'+2}$.
Consider a subgraph $\Gamma'$ spanned by vertices in the set $\bigcup_{j=i+1}^{i'-1}A^i\cup \lk v^i_{i+2}, v^{i'}_{i'+2} \rk$.
Let $m=i'-i$ and let $\ov A^0=\lk v^i_{i+2}\rk$, $\ov A^{m}=\lk v^{i'}_{i'+2} \rk$, and $\ov A^j=A^{i+j}$, for every $j=1,2,\ldots,m-1$.
Then, by Claim 1., the family $\lk \ov A^j \rk$ and the graph $\Gamma'$ satisfy the hypotheses of the lemma.
Thus by the induction (on $k$ --- observe that $m+2=i'-i+1\leqslant k$) assumptions there exists a cycle $\gamma$
in $\Gamma'$ of length at most $k$ and without a diagonal.
Then $\gamma$ is a cycle in $\Gamma$ with the same properties --- contradiction.
Hence $v^i_{i+2}\nsim  v^{i'}_{i'+2}$ and Claim 2 is proved.

\medskip
To conclude the proof of the lemma, observe that the cycle $(v^0_2,v^1_3,v^2_4,\ldots,v^{k}_0,v^0_2)$ is, by Claim 2., the required cycle without a diagonal in $\Gamma$.
\kon

\begin{lem}[Loc.\ $k$--large thickening]
 \label{k-l.thick}
Let $Y$ be a locally $k$--large simple cell complex, for some $k\geqslant 4$.
Then $Th(Y)$ is also locally $k$--large.
\end{lem}
\dow
We have to study links of vertices in $Th(Y)$. Let $v\in Th(Y)$ be a vertex.
Let, for a vertex $w\in Th(Y)_v$, the set $A^w\subseteq Y_v^{(0)}$ (here we identify the $0$--skeleton of the link of a vertex $v$ in a cell complex with the set of vertices joined with the vertex $v$) be the set of all vertices of $Y_v^{(0)}\subseteq Y$ belonging to the minimal cell containing $v$ and $w$.
\medskip

First we prove that $Th(Y)_v$ is flag (the case $k=4$). Let $A\subseteq Th(Y)^{(0)}$ be a finite set of pairwise connected (by edges in $Th(Y)$) vertices.
Then, by the definition of the thickening we have the following.
For any two $w,w'\in A$, and for every $z\in A^w$ and $z'\in A^{w'}$, vertices $z$ and $z'$ are contained in a common cell of $Y$ (the one containing $v,w$ and $w'$) and thus $\langle z,z'\rangle \in Y_v$.
Hence $\ov A=\bigcup_{w\in A} A^w$ is a set of pairwise connected vertices in $Y_v$ and thus, by flagness of $Y_v$, the set $\ov A$ spans a simplex in $Y_v$. It follows that $\ov A \cup \lk v \rk$ is contained in a cell of $Y$ so that $A$ is contained in the same cell and thus $A$ spans a simplex in $Th(Y)_v$. It proves that links in $Th(Y)$ are flag.
\medskip

Now we prove that $Th(Y)_v$ is $k$--large, $k\geqslant 5$.
We do it by a contradiction. Assume there is an $l$--cycle $c=(w_0,w_1,\ldots,w_{l-1},w_0)$ in $Th(Y)_v$ without a diagonal, for $l<k$. We show that then there exists an $l'$--cycle $c'$ in $Y_v$ without a diagonal, for some $l'\leqslant l$. This contradicts $k$--largeness of $Y_v$.

Let $A^i=A^{w_i}$ for $i=0,\ldots,l-1$.
Since there is no diagonal in $c$, we have that for
for every $i,i'$ with $i-i'\neq 1$ mod $l$, there exist vertices $z^i_j\in A^i$ and $z^j_i\in A^j$ not contained in a common cell (containing $v$) and thus not connected by an edge in $Y_v$.
Thus the subgraph $\Gamma$ of $Y_v$ spanned by $\bigcup A^i$, and the family $\lk A^i \rk$ satisfy the hypotheses of Lemma \ref{combinat}.
By this lemma, there exists $l'\leqslant l$ and an $l'$--cycle $c'$ in $\Gamma$.
\end{proof}

\rem Fr\' ed\' eric Haglund introduced the notion of \emph{face complex} of a cell complex (see \cite[Section 1]{JS3}). Vertices in the face complex correspond to cells in the original complex and span a simplex whenever correspond to cells contained in a common cell. Hence, the face complex is a full subcomplex of the barycentric subdivision of the thickening. Haglund proved that the face complex of a simplicial complex $X$ is $k$--large iff $X$ is $k$--large; see \cite[Appendix B]{JS3}. This implies the above Lemma~\ref{k-l.thick} in the case of cubical complexes in view of the following useful result of Jaros\l aw Weksej (whose immediate proof we leave to the reader).

\begin{prop}[Lemma 3.32 in \cite{Wek}]
\label{weksej}
  Let $Y$ be a cubical complex and let $v$ be its vertex. Then the link $Th(Y)_v$ is isomorphic to the face complex of the link $Y_v$.
\end{prop}

\begin{de}[No-$\Delta$'s]
\label{notr}
A cell complex is called \emph{no-$\Delta$} if three cells intersect whenever they pairwise intersect.
\end{de}

\begin{lem}[Thickening of no-$\Delta$]
\label{notrth}
Let $Y$ be a locally flag no-$\Delta$ cell complex. Then $Th(Y)$ is flag.
\end{lem}
\dow
$Th(Y)$ is locally flag by Lemma \ref{k-l.thick}.
Let $A\subseteq Th(Y)$ be a set of pairwise connected vertices of cardinality at least $4$. Let $v\in A$. By the no-$\Delta$ condition every two vertices in $A\setminus \lk v \rk$, span a simplex with $v$. Thus $A\setminus \lk v \rk$ is a set of pairwise connected, in $Th(Y)_v$, vertices. By local flagness, $A\setminus \lk v \rk$ span a simplex in $Th(Y)_v$ and thus $A$ span a simplex in $Th(Y)$.
\kon

\begin{lem}
\label{borsuk}
Let $Y$ be a cell complex. Then $Th(Y)$ is homotopically equivalent with $Y$.
\end{lem}
\begin{proof}
It follows immediately from Borsuk's Nerve Theorem \cite{Bjorn}.
\end{proof}

\begin{cor}
\label{ksystth}
Let $Y$ be a simply connected locally $k$--large simple cell complex. Then $Th(Y)$ is $k$--systolic.
\end{cor}

\begin{prop}[Loc.\ $5$--large no-$\Delta$ $\Rightarrow$ $SD_2^{\ast}$]
\label{5lnotr}
Let $Y$ be a locally $5$--large no-$\Delta$ simple cell complex.
Then $Th(Y)$ satisfies the $SD_2^{\ast}$ property. Moreover, $Th(\widetilde Y)$ is weakly systolic, for $\widetilde Y$ being the universal cover of $Y$. In particular groups acting geometrically by automorphisms on $\widetilde Y$ are weakly systolic.
\end{prop}
\dow
By Lemma \ref{notrth}, the simplicial complex $Th(Y)$ is flag, and by Lemma \ref{k-l.thick} it is locally $5$--large. It follows that the condition (a) of Definition \ref{sd2*} (of the $SD_2^{\ast}$ property) is satisfied. We then turn to the condition (b).

Let $\widehat W=(v_0;v_1,\ldots,v_5;t)$ be a $5$--wheel with a pendant triangle in $Th(Y)$.
Let $c$ be the cell (in $Y$) containing $v_0,v_1,v_2$.
Let $c_1,c_2,c_3$ be cells of $Y$ containing, respectively, $\lk v_0,v_2,v_3\rk,\lk v_0,v_1,v_5\rk,\lk v_1,v_2,t \rk$. Then, by no-$\Delta$ condition, the faces $c_1 \cap c, c_2 \cap c, c_3 \cap c$ of $c$ intersect in a vertex $w$. Since there is no $4$--wheels in $Th(Y)$, the cycle $(w,v_3,v_4,v_5,w)$ cannot be full and thus $w\sim v_4$ (in $Th(Y)$). It follows that $\widehat W\subseteq B_1(w,Th(Y))$ and thus the lemma is proved.

Weak systolicity follows from the fact that, by Lemma \ref{borsuk}, $Th(\widetilde Y)$ is simply connected and thus, by Theorem \ref{logl}, it is weakly systolic.
\kon

\begin{lem}[Thickening of CAT(0) c.c.]
 \label{flag thick}
Let $Y$ be a simply connected locally flag (i.e.\  loc.\ $4$--large) cubical complex. Then
$Th(Y)$ is a no-$\Delta$ cell complex.
\end{lem}
\begin{proof}
This is a reformulation of the \emph{clique Helly} property of median graphs; cf.\ \cite[Theorem 3.1 and Proposition 3.2]{BaCh}.
\end{proof}

\begin{cor}[Thickening of CAT(-1) c.c.]
\label{cat-1cc}
Let $Y$ be a simply connected locally $5$--large cubical complex (i.e.\ CAT(-1) cubical complex). Then $Th(Y)$ is weakly systolic and groups acting geometrically by automorphisms on $Y$ are weakly systolic.
\end{cor}
\dow
It follows directly from Lemma \ref{k-l.thick}, Lemma \ref{flag thick}, and Proposition \ref{5lnotr}.
\kon

\rems 1) Proposition \ref{5lnotr} shows that simply connected locally $5$--large cell complexes are often weakly systolic, if their cells have
no-$\Delta$ of faces. The simplest such complex seems to be a cubical complex. If there are ``$\Delta$ of faces" (like in simplicial complexes) then we need local $6$--largeness (cf.\ Section \ref{syst}).
Simply connected locally $5$--large cell complexes with no-$\Delta$ of faces exist and groups acting on them geometrically appears in the literature. E.g.\ many Coxeter groups and their Davis complexes.

2) There are weakly systolic groups acting geometrically on CAT(-1) cubical complexes, that are not systolic groups. For example, right-angled Coxeter groups acting geometrically on the hyperbolic space $\mathbb H^k$, for $k=3,4$. Such groups ``contain asymptotically spheres" and thus are not systolic \cites{JS2,O-ciscg,O-ib7scg,OS}. Moreover, in \cite{O-chcg} we construct examples of weakly systolic, not systolic hyperbolic Coxeter groups in every virtual cohomological dimension. This gives the first non-systolic examples of such groups. The construction bases on tools developed in this section.

\subsection{Right-angled hyperbolic buildings}
\label{build}

In this section we show that for a right-angled hyperbolic building there exists an associated weakly systolic complex --- it's ``thickening" (cf.\ Definition \ref{thbuild}). This gives us new examples of weakly systolic groups: lattices in isometry groups of such buildings.

\subsubsection{Preliminaries on Coxeter groups and buildings}
\label{prebuild}
We adopt here notations from \cites{Davi-b,DO}.
A \emph{Coxeter group} is given by a presentation $W=\langle S| (st)^{m_{st}}; \; s,t\in S\rangle$, where $S$ is a finite set, $m_{st}\in \mathbb{N}\cup \lk \infty \rk$, $m_{st}=m_{ts}$ and $m_{st}=1$ iff $t=s$ (here $(st)^{\infty}$ means no relation). A Coxeter group (or a \emph{Coxeter system} $(W,S)$) is called \emph{right-angled} if $m_{st}\in \lk 1,2,\infty \rk$.
A \emph{special subgroup} $W_T$ of a Coxeter group $W$ is a subgroup generated by a subset $T\subseteq S$. A subset $T\subseteq S$ is called \emph{spherical} if $W_T$ is finite. In that case $W_T$ is called also \emph{spherical}. By $\mathcal S$ we denote the poset (with respect to inclusions) of spherical subsets of $S$.
The poset of all nonempty spherical subsets is an abstract simplicial complex: the \emph{nerve} $L=L(W,S)$ of the Coxeter system $(W,S)$.
The geometric realization of the poset (with respect to inclusions) $\bigcup_{T\in \mathcal{S}} W/W_T$ is called the \emph{Davis complex} and is denoted by $\Sigma=\Sigma(W,S)$.
The \emph{Davis chamber} $K$ is the simplicial cone $\mr {cone}_{b_K}(L')$ over the barycentric subdivision $L'$ of the nerve $L$, with the cone point denoted by $b_K$.
We equip $W$ with a family of equivalence relations $(\sim_s)_{s\in S}$, defined as follows: $w\sim_s w' \Leftrightarrow w'\in \lk w,ws \rk$.

\medskip

For a given Coxeter system $(W,S)$ a  \emph{building of type $(W,S)$} will be denoted $\Phi$ and is defined as follows. $\Phi$ is a set (of \emph{chambers}) equipped with a family of equivalence relations $(\sim_s)_{s\in S}$ and with a family of subsets (called \emph {apartments}) isomorphic to $W$ (here an \emph{isomorphism} is a bijection preserving every $\sim_s$), such that:
\begin{enumerate}
\item
every two chambers are contained in a common apartment;
\item
if two chambers $x,y$ are both contained in apartments $A,A'$, then there exists an isomorphism $A\to A'$ fixing $x$ and $y$;
\item
if apartments $A,A'$ contain a chamber $x$ and both intersect an equivalence class $R$ of $\sim_s$, then there exists an isomorphism $A\to A'$ fixing $x$ and mapping $R\cap a$ to $R\cap A'$.
\end{enumerate}

For a subset $T\subseteq S$ and $x\in \Phi$ we define the \emph{residue} $\mr{Res}(x,T)$ as the set of all $y\in \Phi$ such that there exists a sequence
$(s_1,\ldots,s_l)$ of elements of $T$ and $x_i\in \Phi$ with $x_0=x\sim_{s_1} x_1 \sim_{s_2} x_2 \sim_{s_3} \cdots \sim_{s_l} x_l=y$.

We consider here the following geometric realization $|\Phi|$ of the building $\Phi$. For a point $p$ in a Davis chamber $K$, let $S(p)=\lk s\in S|\; p\in B_1(s,L')\rk$. Observe that $S(p)\neq \emptyset$ iff $p\in L$. Let $|\Phi|=\Phi\times K / \sim$, where $(x,p)\sim (y,q)$ iff $p=q$ and $x\in \mr{Res} (y,S(p))$. In this construction chambers correspond to copies of $K$.
In the sequel we do not distinguish a building $\Phi$ from its geometric realization $|\Phi|$, and both are denoted by $\Phi$.
A building is \emph{right-angled} (respectively \emph{hyperbolic}) if the corresponding Coxeter system is right-angled (respectively Gromov hyperbolic).
Observe that $W$ itself (respectively $\Sigma$) is a building (respectively a geometric realization of a building). By a theorem of Moussong (see \cite[Corollary 12.6.3]{Davi-b}), a right-angled Coxeter group $(W,S)$ is Gromov hyperbolic iff its nerve $L=L(W,S)$ is a $5$--large simplicial complex.
In that case there is a locally $5$--large cubulation of $\Sigma$ (i.e.\ CAT(-1) cubulation), but we do not know any such cubulation for a general right-angled hyperbolic building.

\subsubsection{Thickenings of buildings are weakly systolic}

Vertices in $K$ being barycenters of maximal simplices in $L$ will be called \emph{vertices of $\Phi$}. For a vertex $v$ of $\Phi$, its \emph{star}, denoted $\mr {St}_v$, is the union of all simplices of $\Phi$ (in the triangulation corresponding to the barycentric subdivision of $L$) containing $v$.
\medskip

\rem In the sequel we usually denote stars by $c,c_i$ etc. It follows from the fact that they play a role of ``cells" (as in Section \ref{cc-1}) in our Definition \ref{thbuild} of a ``thickening", analogous to the thickening of a cell complex --- Definition \ref{thick}.

\begin{lem}
\label{starcub}
Let $\pi \colon \Phi\to \Sigma$ be a folding map. Then, for a vertex $v$ of $\Phi$, the image $\pi (\mr {St}_v)$ of its star is a cube of the standard cubulation of $\Sigma$ (cf.\ \cite[Chapter 1.2]{Davi-b}).
\end{lem}
\begin{proof}
If $\Phi=\Sigma$ (i.e.\ if the equivalence class of each $\sim_s$ has two elements) then stars of vertices are cubes of the standard cubulation of $\Sigma$. In the general case observe that $\pi (\mr{St}_v)=\mr {St}_{\pi(v)}$.
\end{proof}

\begin{lem}
\label{staru}
Let $\Phi$ be a right-angled hyperbolic building and let $c_1,c_2$ be two of its stars, with $c_i=\mr {St}(v_i)$, for a vertex $v_i$ of $\Phi4$; $i=1,2$. Let $\sigma_i$ be the maximal simplex in $L$ with barycenter $v_i$ and let $A$ be the set of Davis chambers containing $v_1$ and $v_2$.
\begin{enumerate}
\item $\Phi$ is the union $\bigcup_v \mr{St}_v$ of stars over all vertices of $\Phi$.
\item
If, for a Davis chamber $K$, we have $c_i\cap K\neq \emptyset$, then $v_i\in K$.
\item
$A=\emptyset$ iff $c_1\cap c_2=\emptyset$.
\item
If $A\neq \emptyset$ and $\sigma_1\cap \sigma_2\neq \emptyset$ then $c_1\cap c_2=\bigcup_{K\in A} \mr{cone}_{b_K}(\sigma_1\cap \sigma_2)$.
\item
If $A\neq \emptyset$ and $\sigma_1\cap \sigma_2= \emptyset$, then $A=\lk K \rk$, and $c_1\cap c_2=b_K$, for some Davis chamber $K$.
\end{enumerate}
\end{lem}
\begin{proof}
(1)
Every vertex of $L$, as well as the cone point in $K$ belong to cone over some maximal simplex in $L$. Thus  $\Phi=\bigcup_v \mr{St}_v$.

(2), (3) and (4) are obvious.

(5) follows from the fact that the intersection of two distinct chambers has diameter at most two in corresponding barycentric subdivisions of nerves.
\end{proof}

\begin{lem}[no-$\Delta$ of stars]
\label{notrst}
A right-angled hyperbolic building has \emph{no-$\Delta$ of stars}, i.e.\ if three stars $c_1,c_2,c_3$ pairwise intersect, then they all intersect.
\end{lem}
\begin{proof}
Let $c_1,c_2,c_3$ be three pairwise intersecting stars. By Lemma \ref{staru}, there exists a chamber $K$ with $b_K\in c_1\cap c_2$. Let $\pi_K \colon \Phi\to \Sigma$ be the $K$--based folding map.
Observe that, by Lemma \ref{starcub}, $\pi_K (c_1), \pi_K (c_2),\pi_K (c_3)$ are three pairwise intersecting cubes in the CAT(-1) cubical complex $\Sigma$. Thus, by Lemma \ref{flag thick}, $\pi_K (c_1)\cap \pi_K (c_2)\cap \pi_K (c_3)\neq \emptyset$. By our choice of $K$, we have $\pi_K^{-1}(\pi_K (c_i))=c_i$, for $i=1,2$, and $\pi_K^{-1}(\pi_K (c_1)\cap \pi_K (c_2))=c_1\cap c_2$.
Thus, if $c_3\cap (c_1\cap c_2)=\emptyset$ then $c_3 \cap \pi_K^{-1}(\pi_K (c_1)\cap \pi_K (c_2))=\emptyset$ that implies $\pi_K(c_3) \cap (\pi_K (c_1)\cap \pi_K (c_2)) =\emptyset$; contradiction.
\end{proof}

\begin{de}[Thickening]
\label{thbuild}
A \emph{thickening} $Th(\Phi)$ of a building $\Phi$ is a flag simplicial complex defined as follows. Vertices of $Th(\Phi)$ are cone points $b_K$, for all chambers $K$. The set of vertices span a simplex if they are all contained in a common star of $\Phi$.
\end{de}

\begin{lem}
\label{coflag}
Let $A$ be a finite collection of simplices in a flag simplicial complex $X$, such that every two simplices in $A$ are contained in a common simplex of $X$. Then there exists a simplex containing all simplices from $A$.
\end{lem}
\begin{proof}
Let $A^{(0)}=\bigcup_{\sigma \in A}\sigma^{(0)}$. Every two vertices in $A^{(0)}$ are connected by an edge in $X$ since the two simplices in $A$ containing them are themselves contained in a common simplex. Thus, by flagness, $\lan A^{(0)} \ran$ is a simplex containing all simplices in $A$.
\end{proof}

\begin{lem}[$Th(\Phi)$ loc.\ $5$--large]
\label{l5lb}
The thickening $Th(\Phi)$ of a hyperbolic right-angled building $\Phi$ is locally $5$--large.
\end{lem}
\begin{proof}
Let $b_K$ be a vertex of $Th(\Phi)$, with $K=\mr{cone}_{b_K}(L')$. We check $5$--largeness of $Th(\Phi)_{b_K}$.
For every $b_{K'}\in Th(\Phi)_{b_K}$ let $L_{K'}$ be the minimal simplex in $L$ such that $K'\in \mr{Res}(K,L_{K'})$.

First we prove flagness.
Let $\lk b_{K_i} \rk_{i\in I} $ be a set of vertices pairwise connected by edges.
Then the collection $A=\lk L_{K_i} \rk_{i\in I}$ has the following property. For every $i,j\in I$, simplices $L_{K_i}$ and $L_{K_j}$ are contained a common simplex --- the maximal simplex of $L$ corresponding to a star containing $b_{K_i}$ and $b_{K_j}$. Then, by Lemma \ref{coflag} there is a maximal simplex containing all simplices in $A$, i.e.\ there exists a star containing each $b_{K_i}$. Thus $\lk b_{K_i} \rk_{i\in I} $ spans a simplex in $Th(\Phi)_{b_K}$.

Now we show that there is no full $4$--cycles in $Th(\Phi)_{b_K}$. Suppose that $(b_{K_0},b_{K_1},b_{K_2},b_{K_3},b_{K_0})$ is such a cycle. Since it has no diagonals, for every $i\in\lk 0,1,2,3 \rk$, there exists a vertex $v_i \in L_{K_i}$ with $\langle v_i,v_{i+2\;(\mr{mod}\; 4)}\rangle \notin L$.
Then the cycle $(v_0,v_1,v_2,v_3,v_0)$ is a full cycle in $L$, which contradicts the $5$--largeness of $L$. It implies that there are no full $4$--cycles in
$Th(\Phi)_{b_K}$ and thus $Th(\Phi)$ is locally $5$--large.
\end{proof}

\begin{lem}
\label{flthb}
$Th(\Phi)$ is flag and it is homotopically equivalent to $\Phi$. In particular it is contractible.
\end{lem}
\begin{proof}
Flagness follows from Lemma \ref{notrst} and Lemma \ref{l5lb} in the same way as in the proof of Lemma \ref{notrth}.

For the homotopy type of $Th(\Phi)$ observe that, by Lemma \ref{staru}, intersections of stars are contractible. Thus, as in Lemma \ref{borsuk}, we can apply directly the Borsuk Nerve Theorem \cite{Bjorn} to get the homotopy equivalence.
\end{proof}

\begin{prop}[$Th(\Phi)$ is weakly systolic]
\label{wsbuild}
Let $\Phi$ be right-angled hyperbolic building. Then $Th(\Phi)$ is weakly systolic. In particular uniform lattices in the isometry group of $\Phi$ are weakly systolic.
\end{prop}
\begin{proof}
By Lemma \ref{notrst}, $\Phi$ has no $\Delta$ of stars.
The proof of the lemma is thus analogous to the one of Proposition \ref{5lnotr} --- we use Lemma \ref{l5lb} instead of Lemma \ref{k-l.thick}, and Lemma \ref{flthb} instead of Lemmas \ref{notrth} \& \ref{borsuk}.
\end{proof}

\subsection{Some other (non-)examples}
\label{othex}

\subsubsection{Subgroups}
\label{subgps}
In Subsection \ref{qconv} we show that quasi-convex subgroups of some hyperbolic weakly systolic groups are weakly systolic; cf.\ Corollary \ref{qcsbgp}. In Subsection \ref{sfps} we prove that finitely presented subgroups of some weakly systolic groups are weakly systolic; cf.\ Theorem \ref{fps}. We do not know at the moment whether finitely presented subgroups of every weakly systolic group are weakly systolic. This is true for systolic groups; cf.\ \cite{Wis} and Subsection \ref{fps} below.

\subsubsection{Products}
\label{prod}
Several constructions of ``free products" of $SD_2^{\ast}$ complexes resulting in an $SD_2^{\ast}$ complex are possible. For example, if $X,Y$ are $SD_2^{\ast}$ complexes and $\sigma \in X$, $\tau \in Y$ their maximal simplices, then the following amalgamated union $X'\cup_c Y'$ is an $SD_2^{\ast}$ complex. We set $X'=X\cup \mr{cone}_c(\sigma)$ and $Y'=Y\cup \mr{cone}_c(\tau)$, where $\mr{cone}_c(\sigma)$ denotes the cone over $\sigma$ with the cone point $c$.

However, in general we do not know whether a free product of weakly systolic groups with amalgamation over a finite subgroup is weakly systolic. A result like that holds for systolic groups; cf.\ \cite{OCh}.

\subsubsection{Non-examples}
\label{nonex}
In \cite{O-chcg} we construct weakly systolic groups of arbitrarily large virtual cohomological dimension. However, all the examples constructed there ``contain asymptotically" spheres of dimension at most $3$; see a discussion in \cite{O-chcg}. A conjecture by Januszkiewicz-\' Swi\polhk atkowski (personal communication) states that simply connected locally $5$--large cubical complexes ``contain asymptotically" spheres of dimension at most $3$. We do not know if this is true for weakly systolic groups (observe that systolic groups do not ``contain asymptotically" spheres above dimension one \cites{JS2,O-ciscg,O-ib7scg,OS}).

On the other hand it seems plausible that $\mathbb Z^3$ is a simple example of a (CAT(0) cubical) group that is not weakly systolic.

\section{A combinatorial negative curvature}
\label{neg}
In this section we introduce and study the local condition, $SD_2^{\ast}(7)$ (cf.\ Definition \ref{sd2*k} below), that is a combinatorial analogue of the negative curvature for weakly systolic complexes. In particular, in Theorem \ref{hyp}, we prove that groups acting geometrically on simply connected complexes satisfying that condition are Gromov hyperbolic.
\medskip

However, we begin with a general result concerning weakly systolic complexes and hyperbolicity.

\subsection{Flats vs. hyperbolicity}
\label{flats}

Recall that the \emph{systolic plane} is a $2$--dimensional simplicial complex isomorphic to the triangulation of the Euclidean plane by equilateral triangles. The following result is an analogue of \cite[Theorem 1.2]{Prz-noflat}.

\begin{tw}[Hyperbolicity $\equiv$ no-flats]
\label{noflat}
Let a group $G$ act geometrically by automorphisms on a weakly systolic complex $X$. Then $G$ is Gromov hyperbolic iff there is no isometric embedding of the systolic plane in $X$.
\end{tw}
\dow
The proof follows the idea of the proof of \cite[Theorem 1.2]{Prz-noflat}.
It is clear that If there is an embedding as in the lemma, then $X$ is not Gromov hyperbolic. Thus we now prove the converse.
\medskip

Assume that $X$ is not Gromov hyperbolic. Then, by the Papasoglu's criterion \cite{Papa}, for every $k\geqslant 2$, there exist vertices $v,w$ and $1$--skeleton geodesics $\gamma=(v_0=v,v_1,\ldots,v_n=w)$ and $\gamma'=(v_0'=v,v_1',\ldots,v_n'=w)$ with the following properties.
There exists $i$ such that $d(v_i,v_i')\geqslant k$.
Let $\delta=(z_0^0=v_i,z_1^0,\ldots,z_l^0=v_i')$ be a geodesic. By convexity of balls we have that $\delta \subseteq B_i(v,X)\cap B_{n-i}(w,X)$.
Let $z_j^1$, for $j=0,\ldots,n-1$, be a vertex in $\pi_v(\langle z_j^0, z_{j+1}^0 \rangle)$.
\medskip

\noindent
{\bf Claim 1.} $z_{j+1}^1\neq z_j^1$.
\medskip

\noindent
{\bf Proof of Claim 1.} Assume that $z_{j+1}^1= z_j^1$. Let $s$ and $t$ be vertices in, respectively, $\pi_w(\langle z_j^0, z_{j+1}^0 \rangle)$ and
$\pi_w(\langle z_{j+1}^0, z_{j+2}^0 \rangle)$. Then $s\neq t$, $s\nsim z_{j+2}^0$ and $t\nsim z_{j}^0$ because there are no $4$--wheels in $X$.
Thus there is a $5$--wheel with a pendant triangle $\widehat W=(z_{j+1}^0;s,t,z_{j+2}^0,z_{j}^1,z_{j}^0;u)$, where $u$ is a vertex in
$\pi_w(\langle s,t \rangle)$. Then, by the property $SD_2^{\ast}$, there exists a vertex $x$ with $\widehat W\subseteq B_1(x,X)$. But this is a contradiction, since $d(z_j^1,u)=3$. This finishes the proof of Claim 1.
\medskip

Observe that $z_j^1\sim z_j^0,z_{j+1}^0,z_{j+1}^1$. Inductively we define $z_j^m$ as a vertex in $\pi_v(\langle z_j^{m-1}, z_{j+1}^{m-1} \rangle)$, for $m=1,2,\ldots,l$ and $j=0,\ldots,l-m$.
Then $z_j^m\sim z_j^{m-1}, z_{j+1}^{m-1}, z_{j+1}^{m}$, and $z_j^m\nsim z_{j-1}^{m-1}, z_{j+2}^{m-1}$.
The full subcomplex $\Delta_l$ spanned by the set $\lk z_j^m \rk$ is isomorphic to an equilateral triangle on the systolic plane with the side of length $l$.
\medskip

\noindent
{\bf Claim 2.} The complex $\Delta_l$ is isometrically embedded in $X$.
\medskip

\noindent
{\bf Proof of Claim 2.}
We have to show that for every two vertices $z,u\in \Delta_l$ their distance $d_{\Delta_l}(z,u)$ in $\Delta_l$ is the same as their distance $d(z,u)$ in $X$.

Let $z=z_j^m$ and $u=z_{j'}^{m'}$.
Assume, by contrary, that there exist pairs $(z,u)$ with $d(z,u) < d_{\Delta_l}(z,u)$. Let a pair $(z,u)$ has the smallest $m'+m$ among such pairs. Observe that if $m=m'=0$ then $z,u$ lie on the geodesic $\delta$ and thus $d(z,u) = d_{\Delta_l}(z,u)$.
W.l.o.g. assume that $m+j\leqslant m'+j'$ and $m\leqslant m'$ (the other cases can be treated analogously by symmetry).

If $m=m'$ then consider the pair $(z'=z_j^{m-1},u)$. We have $d(z',u)\leqslant d(z,u) +1 < d_{\Delta_l}(z,u)+1=d_{\Delta_l}(z',u)$. This contradicts our assumptions on $(z,u)$.

Thus we assume that $m<m'$.
If $j'\leqslant j$ then $d_{\Delta_l}(u,z)=m'-m=d(u,z)$; contradiction. Hence $j< j'$. Consider the pair $(z,u')$ with $u'=z_{j'+1}^{m'-1}$. By our assumptions we have
\begin{align*}
d(z,u') &= d_{\Delta_l}(z,u')=(j'+1-j)+(m'-1-m)\\
& =(j'-j)+(m'-m)=d_{\Delta_l}(z,u)\geqslant d(z,u)+1.
\end{align*}
Thus there exists
a geodesic in $X$ between $z$ and $u'$ passing through $u$. But this contradicts the convexity of the ball $B_{(n-i)+(m'-1)}(w,X)$ containing $z$ and $u'$.
This finishes the proof of Claim 2.
\medskip

Thus, for every $l$ there exists a systolic equilateral triangle $\Delta_l$ with the side $l$ isometrically embedded in $X$. Since $X$ admits the geometric $G$--action it follows, cf.\ \cite[Lemma 3.4]{Prz-noflat}, that there is an isometrically embedded systolic plane in $X$.
\kon

\rem An analogous result holds also for CAT(0) groups; cf.\ \cite[Chapter III.$\Gamma$, Theorem 3.1]{BrHa}.

\subsection{Negative curvature}
\label{nega}
\begin{de}[$SD_2^{\ast}(k)$ property]
\label{sd2*k}
For $k\geqslant 5$, a flag simplicial complex $X$ satisfies the \emph{$SD_2^{\ast}(k)$ property} (or is an \emph{$SD_2^{\ast}(k)$ complex}) if the following two conditions hold.

(a) $X$ does not contain $4$--wheels,

(b) for $5\leqslant l<k$ and every $l$--wheel with a pendant triangle $\widehat W$ in $X$, there exists a vertex $v$ with $\widehat W\subseteq B_1(v,X)$.
\end{de}

Observe that the condition
$SD_2^{\ast}(6)$ is equivalent to $SD_2^{\ast}$, and that
$SD_2^{\ast}(k)$ implies $SD_2^{\ast}(l)$, for $l\leqslant k$.

\begin{lem}[Strict geodesic contraction]
\label{strcon}
Let $v$ be a vertex of a simply connected $SD_2^{\ast}(7)$ complex $X$.
Let $n\geqslant 2$ and let $w_1\sim w_2$ be two vertices on the sphere $S_n(v,X)$.
Then $\pi_v^2(w_i)=\pi_v(\pi_v (w_i))\subseteq \pi_v^2(w_j)=\pi_v(\pi_v (w_j))$ for some $i\neq j$.
\end{lem}
\dow
By Theorem \ref{logl}, $X$ is a weakly systolic complex.
If $\pi_v(w_j)\subseteq \pi_v(w_i)$ then $\pi_v^2(w_i)\subseteq \pi_v^2(w_j)$ and we are done. If $n=2$ then $\pi_v^2(w_1)=\pi_v^2(w_2)=v$. Thus for the rest of the proof we assume that $n\geqslant 3$ and  $\pi_v(w_j)\nsubseteq \pi_v(w_i)$ for any $i\neq j$.
Then the following vertices exist: $w_0\in \pi_v(\langle w_1,w_2\rangle)$; $w_3 \in \pi_v(w_2)\setminus \pi_v(w_1)$; $w_6 \in \pi_v(w_1)\setminus \pi_v(w_2)$.

We argue by contradiction. Suppose that $\pi_v^2(w_i)\nsubseteq \pi_v^2(w_j)$ for any $i\neq j$; $i,j=1,2$. Then we can find vertices:
$w_4\in \pi_v(\langle w_0,w_3\rangle)\setminus \pi_v(\langle w_0,w_6\rangle)$; $w_5\in \pi_v(\langle w_0,w_6\rangle)\setminus \pi_v(\langle w_0,w_3\rangle)$; $t\in \pi_v(\langle w_4,w_5\rangle)$.
Observe that then there exists a $6$--wheel with a pendant triangle $\widehat W=(w_0;w_4,w_5,w_6,w_1,w_2,w_3;t)$ in $X$.
By the condition $SD_2^{\ast}(7)$, there is a vertex $z$ with $\widehat W\subseteq B_1(z,X)$. But this contradicts the fact that $d(w_1,t)=3$. Hence $\pi_v^2 (w_i)\subseteq \pi_v^2 (w_j)$ for some $i\neq j$; $i,j=1,2$.
\kon

\begin{lem}[Thin bigons]
\label{bigons}
Let $v,w$ be two vertices in a simply connected $SD_2^{\ast}(7)$ complex $X$, and let $\gamma=(v_0=v,v_1,\ldots,v_n=w)$ be a $1$--skeleton geodesic between them. For any other geodesic $\gamma'=(v_0'=v,v_1',\ldots,v_n'=w)$ from $v$ to $w$
we have $d(v_i,v_i')\leqslant 1$, for every $i=0,1,\ldots,n$.
\end{lem}
\dow
Assume that there exists $i$ such that $d(v_i,v_i')=2$. Let moreover $i$ be the biggest number with this property. Then it is clear that $d(v_{i+1},v_{i+1}')=1$. By Theorem \ref{strcon}, we have that, w.l.o.g.
$\pi_v(\pi_v(v_{i+1})) \subseteq \pi_v(\pi_v(v_{i+1}')$. Thus there exists a vertex $z'$ in $S_{i-1}(v,X)$ such that $z' \sim v_i, v_{i}', z$, for some $z\in \pi_v(\langle v_{i+1},v_{i+1}'\rangle)$. Observe that for a vertex $t\in \pi_w(\langle v_{i+1},v_{i+1}'\rangle)$, we have that there exists a $5$--wheel with a pendant triangle $\widehat W=(z;v_{i+1},v_{i+1}',v_i',z',v_i;t)$. Then, by the condition $SD_2^{\ast}$ we have that there exists a vertex $z_0$ with $\widehat W\subseteq B_1(z_0,X)$. But this contradicts the fact that $d(t,z')=3$. Thus the lemma follows.
\kon

\begin{tw}[$SD_2^{\ast}(7)$ implies hyperbolicity]
\label{hyp}
The universal cover of an $SD_2^{\ast}(7)$ complex is Gromov hyperbolic.
In particular, groups acting geometrically by automorphisms on simply connected $SD_2^{\ast}(7)$ complexes are Gromov hyperbolic.
\end{tw}
\dow
This follows immediately from Lemma \ref{bigons}, by Papasoglu's criterion \cite{Papa}.
\kon

\rem Another proof of hyperbolicity could be similar to the one of hyperbolicity of $7$--systolic groups given in \cite{JS1}. This would use filling diagrams; cf.\ e.g.\ \cite{OCh}.

\subsection{Examples}
\label{exneg}

Here we provide some examples of weakly systolic $SD_2^{\ast}(7)$ complexes. In particular they include $7$--systolic groups (cf.\ Proposition \ref{7syst}) and CAT(-1) cubical groups (cf.\ Corollary \ref{cat-1c}); compare also Section \ref{syst}.

The following is a direct consequence of Definition \ref{sd2*k}.
\begin{prop}
\label{7syst}
A locally $k$--large complex is an $SD_2^{\ast}(k)$ complex for $k\geqslant 6$.
\end{prop}

The proof of the following is the same as the one of Proposition \ref{5lnotr}.

\begin{prop}[Loc.\ $5$--large no-$\Delta$ $\Rightarrow$ $SD_2^{\ast}(k)$]
\label{5lnotrn}
Let $Y$ be a locally $5$--large no-$\Delta$ cell complex.
Then $Th(Y)$ satisfies the $SD_2^{\ast}(k)$ property, for every $k\geqslant 6$. \end{prop}

As a corollary we get the following strengthening of Corollary \ref{cat-1cc}.

\begin{cor}[Thickening of CAT(-1) c.c.]
\label{cat-1c}
Let $Y$ be a simply connected locally $5$--large cubical complex (i.e.\ CAT(-1) cubical complex). Then $Th(Y)$ is weakly systolic
$SD_2^{\ast}(k)$ complex, for every $k\geqslant 6$.
\end{cor}

Similarly, following the proof of Proposition \ref{wsbuild}, we obtain the following.

\begin{prop}[$\mr {Th}(\Phi)$ satisfies $SD_2^{\ast}(k)$]
\label{sdkbuild}
Let $\Phi$ be a right-angled hyperbolic building.
Then $Th(Y)$ satisfies the $SD_2^{\ast}(k)$ property, for every $k\geqslant 6$. \end{prop}

\subsection{Gromov boundary}
\label{bdry}
Let $v$ be a vertex of a weakly systolic complex $X$.
Then, for every $n>0$, we can define a map $\pi_v \colon ((S_{n+1}(v,X))')^{(0)} \to ((S_{n}(v,X))')^{(0)}$ as follows. For a barycenter $w$ of a simplex $\sigma$ in $S_{n+1}(v,X)$, its image $\pi_v(w)$ is the barycenter of $\pi_v(\sigma)$ (here $\pi_v$ as in Definition \ref{sdn(A)new}). There should be no confusion between this new definition of $\pi_v$ and the one from Definition \ref{sdn(A)new}.

Let $v$ be a vertex of an $SD_2^{\ast}(7)$ complex $X$.
By Lemma \ref{strcon}, we have immediately the following.

\begin{lem}
\label{newpi}
Let $X$ be a simply connected $SD_2^{\ast}(7)$ complex $X$.
Then for every vertex $v\in X$ and for every $n>0$ the map
$\pi_v^2=\pi_v \circ \pi_v \colon ((S_{n+2}(v,X))')^{(0)} \to ((S_{n}(v,X))')^{(0)}$extends to a simplicial map
$\pi_v^2 \colon (S_{n+2}(v,X))' \to (S_{n}(v,X))'$.
\end{lem}

\begin{tw}[Gromov boundary]
\label{bdryinv}
The Gromov boundary of an $SD_2^{\ast}(7)$ complex $X$ is homeomorphic to the inverse limit $\mr {inv \; lim}_{n\to \infty}\lk (S_{2n}(v,X))', \pi_v^2\rk$, for every vertex $v\in X$.
\end{tw}
\dow
The proof is the same as the one of \cite[Lemma 4.1]{O-ib7scg}.
\kon

\rem Theorem \ref{bdryinv} above is an analogue of \cite[Lemma 4.1]{O-ib7scg}.
The previous result appeared to be very useful for studying the topology of the boundary of $7$--systolic groups and its algebraic consequences; cf.\ \cites{O-ib7scg, Sw-propi, Zaw}.
The new result (Theorem \ref{bdryinv}) provides a nice combinatorial description of the Gromov boundary for some more classical groups, e.g.\ CAT(-1) cubical groups (cf.\ Corollary \ref{cat-1c}) and some hyperbolic buildings (cf.\ Proposition \ref{sdkbuild}). We believe it can be useful for various purposes.

\subsection{Quasi-convex subgroups}
\label{qconv}

In this subsection we prove analogues of some results of \cite{HS}. The goal is to show that quasi-convex subgroups of groups acting geometrically on weakly systolic $SD_2^{\ast}(7)$ complexes are themselves acting geometrically on weakly systolic $SD_2^{\ast}(7)$ complexes --- Corollary \ref{qcsbgp}. This provides new examples of weakly systolic groups.

\begin{lem}[Y-lemma; cf.\ {\cite[Lemma 5.1]{HS}}]
\label{Ylem}
Let $v$ be a vertex in a weakly systolic complex $X$. Let $v_1, v_2$ be vertices at distance $n$ from $v$ and with $d(v_1,v_2)=g\leqslant n$. Then there is
a geodesic of length $n-d$ with starting at $v$ that extends to a geodesic to either of the vertices $v_i$.
\end{lem}
\dow
The proof is the same as the one of \cite[Lemma 5.1]{HS} --- one uses Lemma \ref{ballvert} instead of \cite[Corollary 4.10]{HS}.
\kon

\begin{lem}[cf.\ {\cite[Lemma 5.2]{HS}}]
\label{52HS}
Let $X$ be a weakly systolic $SD_2^{\ast}(7)$ complex. Let $(v_0, v_1,\ldots, v_n)$, $(v_0,v_1',\ldots,v_n')$ be two geodesics with the same origin and such that $d(v_1,v_1')=2$. Then $n <5 +d(v_n,v_n')$.
\end{lem}
\dow
If $n\leqslant 5$ then, by thinness of bigons (cf.\ Lemma \ref{bigons}), we are done. Assume that $n\geqslant 6$.
\medskip

\noindent
\emph{(Step 1.)}
We argue, as in \cite[proof of Lemma 5.2]{HS}, by contradiction. Suppose that $d(v_n,v_n')\leqslant n-5$. By Lemma \ref{Ylem}, there is a geodesic $(v_0,u_1,u_2,u_3,u_4,u_5)$ which extends to geodesics from $v_0$ to either of the vertices $v_n$ or $v_n'$. By Lemma \ref{bigons} we have $d(u_i.v_i)\leqslant 1 \geqslant d(u_i,v_i')$ for $i=1,2,3,4,5$.
By weak systolicity we have that $v_2\nsim v_1'$, $v_1\nsim v_2'$ and $v_2\neq v_2'$.
\medskip

\noindent
\emph{(Step 2.)}
Suppose that $u_2\sim v_1$. Then $u_2\nsim v_1'$, $v_2'\sim u_1$, and one of vertices $u_3,v_3'$, say $u_3$, spans a simplex with $\langle u_2.v_2' \rangle$. By the $SD_2^{\ast}$ property there exists a vertex $u'$ such that the $5$--wheel with a pendant triangle $(u_1;u_2,v_2',v_1',v_0,v_1;u_3)$ is contained in $B_1(u',X)$. This contradicts the fact that $d (v_0,u_3)=3$. Thus $u_2\nsim v_1$ and, analogously, $u_2\nsim  v_1'$.
This implies that $v_2\neq u_2\neq v_2'$ and $u_1\sim v_2,v_2'$.
\medskip

\noindent
\emph{(Step 3.)}
Suppose that $v_2\sim v_2'$. If either of the vertices $v_3,u_3,v_3'$, say $v_3$, is joined with $v_2$ and $v_2'$, then we have a $5$--wheel with a pendant triangle $(u_1;v_2,v_2',v_1',v_0,v_1;v_3)$ which contradicts, by the property $SD_2^{\ast}$ as above, the fact that $d(v_0,v_3)=3$.
Thus none of the vertices $v_3,u_3,v_3'$ is joined simultaneously with $v_2$ and $v_2'$. Thus, by the weak systolicity $v_3\nsim v_3'$ (otherwise there is a $4$--cycle $(v_3,v_3',v_2',v_2,v_3)$ that has to have a diagonal) and, in particular $v_3\neq u_3 \neq v_3'$.

Now, proceeding as in Step 1. and Step 2. (replacing indexes $i$ by $i+2$) we get that $u_3\sim v_4,v_4'$, and $v_4\neq v_4' \neq u_4 \neq v_4$, $v_3\nsim v_4'$, and $v_3'\nsim v_4$. Then we have a $5$--wheel $(u_3;u_2,v_3,v_4,v_4',v_3';v_2)$ or a $6$--wheel $(u_3;u_2,v_3,v_4,u_4,v_4',v_3';v_2)$. In both cases, by the property $SD_2^{\ast}(7)$, there exists a vertex $u$ with $u\sim v_2,v_3',v_4$.
By the weak systolicity, considering the cycle $(v_2,u,v_3',v_2',v_2)$ we have $u\sim v_2'$. But this gives a $5$--wheel with a pendant triangle
$(u_1;v_2,v_2',v_1',v_0,v_1;u)$,
which contradicts, by the property $SD_2^{\ast}$ as above, the fact that $d(v_0,u)=3$.
\medskip

\noindent
\emph{(Step 4.)}
We have thus $v_2\nsim  v_2'$. Then, as in Step 1. and Step 2. we conclude
(replacing indexes $i$ by $i+1$) that $u_2\sim v_3$. Thus we have a $6$--wheel with a pendant triangle $(u_1;v_2,u_2,v_2',v_1',v_0,v_1;v_3)$. This contradicts again, by the property $SD_2^{\ast}(7)$, the fact that $d(v_0,v_3)=3$.

This is the final contradiction that finishes the proof.
\kon

\begin{cor}[cf.\ {\cite[Corollary 5.3]{HS}}]
\label{53HS}
Let $v_1, v_2,v$ be vertices of a weakly systolic $SD_2^{\ast}(7)$ complex $X$, and suppose that
$d(v_1, v_2)\leqslant d$ and $d(v_1,x)=d(v_2,x)=n\geqslant d+5$. Denote by $\sigma_1,\sigma_2$ the
projections of the vertex $v$ on the spheres $S_{n-1}(v_1,X)$ and $S_{n-1}(v_2,X)$ respectively.
Then $\sigma_1 \cup \sigma_2$ spans a simplex of $X$.
\end{cor}
\dow
The proof follows verbatim the one of {\cite[Corollary 5.3]{HS}} --- one uses Lemma \ref{52HS} instead of {\cite[Lemma 5.2]{HS}}.
\kon

\begin{de}[Quasi-convexity]
\label{qconvde}
Given $K > 0$, we say that a subcomplex $Y$ in a connected simplicial
complex $X$ is \emph{$K$--quasi-convex} whenever the following holds: for any geodesic $(v_0,\ldots,v_n)$ in $X$ such that $v_0,v_n\in Y$ we have $d(v_i,Y)\leqslant K$ for $i=0,1,\ldots,n$. A
subcomplex $Y$ is \emph{quasi-convex} if it is $K$--quasi-convex for some $K$.

A subgroup $H$ of a Gromov hyperbolic group $G$ is \emph{quasiconvex}
if $H$ is a quasiconvex subset in the Cayley graph $C(G,S)$ for some finite generating set $S$ of $G$.
\end{de}

The following theorem together with its proof are analogues of \cite[Theorem 5.5]{HS} and the proof there.

\begin{tw}[Convex neighborhood of quasi-convex]
\label{qconv>conv}
Let $X$ be a locally finite weakly systolic $SD_2^{\ast}(7)$ complex. Let $Y\subseteq X$ be its $K$--quasi-convex subcomplex, for some $K>0$. Then there exists an integer $n(K)$ such that for every $n\geqslant n(K)$ the ball $B_n(Y,X)$ is convex.
\end{tw}
\dow
The proof follows the one of \cite[Theorem 5.5]{HS}.

$X$ admits \emph{quasi-projections} on quasi-convex subsets (see \cite[proof of Theorem 5.5]{HS}), hence there exists a natural number $d=d(K)$ such that for every vertex $v\in X$ and every vertices $v_1,v_2\in Y$ with $d(v,v_1)=d(v,v_2)=d(v,Y)$ we have $d(v_1,v_2)\leqslant d$.

Set $n(K)=\max \lk K, d(K)+4 \rk$ and observe that for every $n\geqslant n(K)$ the ball $B_n(Y,X)$ is connected by quasi-convexity. Thus, in view of Lemma \ref{lco>co} it is enough to check the local $3$--convexity of $B_n(Y,X)$; cf.\ Definition \ref{3convex}.
Obviously, it is enough to consider only the local convexity at vertices $v$ with $d(v,Y)=n$. Let $v_1\sim v_2 \sim v_3$ be distinct vertices in $X_v$, with $v_1\nsim v_3$ and $v_1,v_3\in B_n(Y,X)$.
Suppose that $v_2 \notin B_n(Y,X)$, i.e.\ $d(v_2,Y)=n+1$. We show that this leads to a contradiction and it will finish the proof.
\medskip

Let $A\subseteq Y$ be the set of all vertices of $Y$ at distance $n+1$ from $v_2$. By the properties of quasi-projection, the diameter of $A$ is at most $d$. We claim that the set of simplices $\lk \pi_{v_2}(w) | \; \; w\in A \rk$
span a simplex $\tau$ in $X$. Indeed, since $d(v_2,Y)=n+1\geqslant d(K)+5$, it follows from Corollary \ref{53HS} that any pair of such simplices span a simplex in $X$. Thus, by flagness and local finite dimensionality, the claim follows.
It follows that $v_2 \in S_n(\tau,X)$ and $v_1,v_3\in S_{n-1}(\tau,X)$.
However, since $n\geqslant 4$ and, by Lemma \ref{bbac}, the ball $B_{n-1}(\tau,X)$ is convex, we have that $v_1\sim v_3$ --- contradiction.
\kon

The next result provides, in particular, other examples of weakly systolic groups.

\begin{cor}[Quasi-convex subgroups]
\label{qcsbgp}
Let $G$ be a group acting geometrically by automorphisms on a weakly systolic $SD_2^{\ast}(7)$ complex $X$. Then any quasi-convex subgroup $H$ of $G$ is convex cocompact, i.e.\ there exists a convex $H$--invariant subcomplex $Y$ on which $H$ acts geometrically. In particular $H$ acts geometrically on a weakly systolic $SD_2^{\ast}(7)$ complex.
\end{cor}
\dow
This follows immediately from Theorem \ref{qconv>conv}; cf.\ the proof of \cite[Corollary 5.8]{HS}. The last assertion follows from Lemma \ref{conv>contr} and Lemma \ref{cover}.
\kon

\section{Weakly systolic complexes with $SD_2^{\ast}$ links}
\label{sec-lasf}

In this section we study a class of weakly systolic complexes, whose asymptotic properties resemble very much the ones of systolic complexes; cf.\ Subsections \ref{coninf} and \ref{SHA}.
Those are the so called weakly systolic complexes with $SD_2^{\ast}$ links.
They provide many new (i.e.\ a priori not systolic) examples of (highly dimensional) groups with interesting asphericity properties; cf.\ Theorem \ref{coninfth}, Theorem \ref{shatw}, remarks afterwards, and \cites{O-chcg,OS}.

\begin{de}
\label{lasf}
A flag simplicial complex $X$ is called a \emph{complex with $SD_2^{\ast}$ links} (respectively a \emph{complex with $SD_2^{\ast}(k)$ links}) if $X$ and every of its links satisfy the property $SD_2^{\ast}$ (respectively $SD_2^{\ast}(k)$); cf.\ Definition \ref{sd2*k}.

\end{de}

The following proposition is the key result in showing various asphericity properties of weakly systolic complexes with $SD_2^{\ast}$ links.

\begin{prop}
\label{fullasf}
Let $X$ be an $n$--dimensional ($n< \infty$) flag simplicial complex.
Then the following three conditions are equivalent.

i) $X$ is a complex with $SD_2^{\ast}(k)$ links.

ii) $X$ does not contain $4$--wheels and full $i$--wheels with pendant triangles (i.e.\ $5$--wheels with pendant triangles being full subcomplexes of $X$), for $i=5,\ldots,k-1$.

iii) Every full subcomplex of $X$ satisfies the $SD_2^{\ast}(k)$ property. In particular every full subcomplex of $X$ is aspherical.
\end{prop}

\begin{proof}
$(i)\Rightarrow (ii)$.
By the definition of the $SD_2^{\ast}(k)$ property $X$ does not contain $4$--wheels.

Assume, by contradiction, that $\widehat W=(v_0;v_1,\ldots,v_i;t)$ is a full $i$--wheel with a pendant triangle in $X$ (cf.\ Definition \ref{wheels}).

Observe that then $v_0\nsim t$.
Since $X$ satisfies the property $SD_2^{\ast}(k)$,
there exists a vertex $w_1\neq v_0$ such that $\widehat W\subseteq X_{w_1}$.
By the assumption on $X$ we have that $X_{w_1}$ satisfies the property
$SD_2^{\ast}(k)$ so that there exists a vertex $w_2\in X_{w_1}$ with
$\widehat W\subseteq (X_{w_1})_{w_2}=X_{\langle w_1,w_2\rangle}$.
We can continue this process until we get vertices $w_1,\ldots,w_{n-1}$ such that
$\widehat W\subseteq X_{\langle w_1,\ldots,w_{n-1}\rangle}$.
However this is a contradiction, since $X_{\langle w_1,\ldots,w_{n-1}\rangle}$ is at most $1$--dimensional and thus cannot contain $\widehat W$.
\medskip

$(ii)\Rightarrow (iii).$
Let $Y$ be a full subcomplex of $X$.
Then $Y$ is a flag complex and we have to check both conditions: (a) and (b) from Definition \ref{sd2*k}.

For (a) observe that if $W\subset Y$ is a $4$--wheel then, since $Y$ is full in $X$, the complex $W$ is a $4$--wheel in $X$, too. Thus, by $(ii)$, $Y$ cannot contain $4$--wheels.

Similarly, for (b), observe that if $\widehat W$ is a full $i$--wheel with a pendant triangle in $Y$, $i=5,\ldots,k-1$, then $\widehat W$ is also a full $i$--wheel with a pendant triangle in $X$. Thus, by $(ii)$, $Y$ does not contain full $i$--wheels with pendant triangles.

The last assertion follows from the fact that complexes satisfying the property $SD_2^{\ast}(k)$ are aspherical --- Theorem \ref{logl}.
\medskip

$(iii)\Rightarrow (i).$ This implication is clear, since $X$, and every its link are full subcomplexes of $X$.
\end{proof}

\rem Observe that the implications $(ii)\Rightarrow (iii) \Rightarrow (i)$ hold also without the assumption about finiteness of dimension.

\begin{cor}[Loc.\ $k$--large complexes]
\label{systsd}
For $k\geqslant 6$, a locally $k$--large complex is a complex with $SD_2^{\ast}(k)$ links.
\end{cor}

\rem There are weakly systolic complexes with $SD_2^{\ast}$ links that are not systolic. In \cite{OS} we provide a construction of such complexes, equipped with a geometric group action, in arbitrarily high (cohomological) dimension.

\subsection{Finitely presented subgroups}
\label{sfps}
Here we prove that finitely presented subgroups of
fundamental groups of complexes with $SD_2^{\ast}(k)$ links act geometrically on weakly systolic complexes with $SD_2^{\ast}(k)$ links; cf.\ Theorem \ref{fps} below. The proof follows almost verbatim
the proof of an	 analogous theorem by Dani Wise \cite[Theorem 5.7]{Wis} concerning the case of torsion-free systolic groups.

A version of the following definition was given in \cite[Definition 5.4]{Wis}.

\begin{de}[Full tower]
\label{fullt}
A map $Y\to X$ of connected flag simplicial complexes is a \emph{full tower} if it can be expressed as the composition
$$
Y=X_n\hookrightarrow \widehat X_{n-1}\to X_{n-1} \hookrightarrow \cdots \hookrightarrow \widehat X_1 \to X_1=X,
$$
where the maps are alternately inclusions of full subcomplexes and covering maps.

Let $f\colon Z\to X$ be a map of connected flag simplicial complexes. A map $g\colon Z\to Y$ is a \emph{full tower lift} of $f$ if there is a full tower $h\colon Y\to X$ such that the following diagram commutes:

$$   \begindc{\commdiag}[3]  \obj(12,1)[a]{$Z$}  \obj(35,1)[b]{$X$} \obj(35,15)[c]{$Y$}  \mor{a}{b}{$f$}  \mor{c}{b}{$h$}  \mor{a}{c}{$g$}  \enddc  $$

The full tower lift $g$ is \emph{compact} if $Y$ is compact and $g$ is \emph{maximal} if for every full tower lift $g'\colon Z\to Y'$ of $g$ the map $Y'\to Y$ is an isomorphism.
\end{de}

The following crucial lemma is a version of \cite[Lemma 5.5]{Wis}. For completeness we provide its proof here. It is the same as the proof of Wise' lemma.

\begin{lem}[Maximal full tower lift]
\label{maxlift}
Let $Z, X$ be flag simplicial complexes and let moreover $Z$ be connected and finite, and $X$ be locally finite.
Then every simplicial map $f\colon Z\to X$ has a maximal compact full tower lift.
\end{lem}
\begin{proof}

We construct the tower inductively as follows. Let $X_1=X$. For $i\geqslant 1$ let $\widehat X_i$ be the based covering space of $X_i$ corresponding to the image of $\pi_1(Z)$ and let $X_{i+1}$ be the span of the image of the lift $Z\to \widehat X_i$.

If no maximal full tower existed than we would have an infinite sequence of immersions $\cdots \to X_{i+1} \to X_i \to \cdots \to X_1$.
Observe that the number of vertices in every $X_i$ is bounded by the number of vertices in $Z$. Thus there is a uniform bound on the number of  simplices in $X_i$'s and hence there are finitely many isomorphism types of $X_i$'s.
Let $X_1,X_2,\ldots,X_M$ represent all the types.

We claim that $X_{i+1}\to X_i$ is injective for $i\geqslant M$.
To show this we observe that $X_{i+1}$ is isomorphic $X_l$, for some $1\leqslant l \leqslant M$, and the map $X_{i+1}\to X_l$ is a combinatorial immersion between compact complexes. By \cite[Lemma 6.3]{Wisa} we have that $X_{i+1}\to X_l$ is an isomorphism and it follows that $X_{i+1}\to X_i$ is injective.

By the claim we have that $\cdots \to X_{M+2}\to X_{M+1} \to X_M$ is a sequence of inclusions of subcomplexes. It has to terminate because of the uniform bound on the number of simplices in $X_i$'s.
\end{proof}

\begin{lem}(Coverings)
\label{cover}
Let $h\colon Y\to X$ be a covering of a complex $X$ satisfying the property $SD_2^{\ast}(k)$, $k\geqslant 6$. Then $Y$ satisfies the property $SD_2^{\ast}(k)$.
\end{lem}
\begin{proof}
By the definition of a covering map $Y$ does not contain $4$--cycles.
Thus, we need only to check that the condition (b) of Definition \ref{sd2*k} is satisfied by $Y$. Let $\widehat W$ be an $i$--wheel with a pendant triangle contained in $Y$, for $i=5,\ldots,k-1$. Then, since $h$ is a covering map, we have that $h(\widehat W)$ is an $i$--wheel with a pendant triangle in $X$. By the $SD_2^{\ast}(k)$ property for $X$, there exists a vertex $v\in X$ with $h(\widehat W)\subseteq B_1(v,X)$. Then $\widehat W\subseteq B_1(v',X)$ for some $v'\in Y$ with $h(v')=v$. This finishes the proof.
\end{proof}

\begin{tw}[Finitely presented subgroups]
\label{fps}
Let $k\geqslant 6$ and let $X$ be a compact complex with $SD_2^{\ast}(k)$ links.
Then every finitely presented subgroup of $\pi_1(X)$ is a fundamental group of a finite complex with $SD_2^{\ast}(k)$ links.
\end{tw}

\begin{proof}
Let $H$ be a finitely presented subgroup of $\pi_1(X)$ and let $f\colon Z\to X$
be a simplicial map of compact complexes such that $\pi_1(Z)=H$, and $f_{\ast}\colon \pi_1(Z)\to \pi_1(X)$ is an isomorphism on $H$.

By Lemma \ref{maxlift} there exists a maximal compact full tower lift
$g\colon Y\to X$ of $f$. By the definition of a full tower, by Proposition \ref{fullasf}, and by Lemma \ref{cover}, we have that $Y$ is a weakly systolic complex with $SD_2^{\ast}(k)$ links.

The map $\pi_1(Z)\to \pi_1(Y)$ is surjective because the tower is maximal and is injective because $f_{\ast}$ factors through it.
\end{proof}

\rem
If, in the considerations above, we replace everywhere ``complex(es) with $SD_2^{\ast}(k)$ links" by ``locally $k$--large complexes", then we get Wise's result together with its proof.

\subsection{Connectedness at infinity}
\label{coninf}

Here we prove analogues of results from \cite{O-ciscg}, in the case of weakly systolic complexes with $SD_2^{\ast}$ links and groups acting on them geometrically.

Recall, cf.\ e.g.\ \cites{O-ciscg,OS}, that for a group $G$ acting geometrically by automorphisms on a simplicial complex $X$, the \emph{$n$--th homotopy groups at infinity vanish},
denoted by $\pi_n^{\infty}(X)=0$ and $\pi_n^{\infty}(G)=0$, iff for every compact $K\subseteq X$ there exists a compact subset $L\supseteq K$ of $X$ with the following property. For every map $f\colon S^n=\partial B^{n+1}\to X\setminus L$, of the $n$--dimensional sphere $S^n$, there exists a map $F\colon B^{n+1}\to X\setminus K$ extending $f$, i.e.\ $F|_{S^n}=f$. We say that $X$ (respectively $G$) is \emph{simply connected at infinity} if $X$ (respectively $G$) has one end and $\pi_1^{\infty}(X)=0$ (respectively $\pi_1^{\infty}(G)=0$).

\begin{tw}[Connectedness at infinity]
\label{coninfth}
Let $G$ be a group acting geometrically on a weakly systolic complex $X$ with $SD_2^{\ast}$ links. Then $\pi_n^{\infty}(X)=0$ and $\pi_n^{\infty}(G)=0$, for all $n \geqslant 2$. Moreover $G$ is not simply connected at infinity.
\end{tw}
\dow
The proof follows verbatim the proofs of \cite[Theorem 3.1 and Theorem 3.2]{O-ciscg}.
\kon

\begin{cor}
\label{exact}
Let $0\to K \to G \to H\to 0$ be a short exact sequence
of infinite finitely presented groups. If $G$ acts geometrically on a weakly systolic complex with $SD_2^{\ast}$ links
then neither $K$ nor $H$ has one end.
\end{cor}
\dow
The proof is the same as the one of \cite[Corollary 3.3]{O-ciscg}.
\kon

\rems
1) Theorem \ref{coninfth} yields serious restriction on groups acting geometrically on weakly systolic complexes with $SD_2^{\ast}$ links.
In particular they cannot be isomorphic to the fundamental groups of closed manifolds covered by $\mathbb R^n$, for $n\geqslant 3$. See \cite{O-ciscg} for other non-examples.

2) In \cite{OS} we prove actually, basing on Proposition \ref{fullasf}, that groups acting geometrically on weakly systolic complexes with $SD_2^{\ast}$ links are asymptotically hereditarily aspherical (shortly AHA); cf.\ \cite{JS2}.
This implies the results about connectedness at infinity
above and has many other consequences; cf.\ \cite{OS}.

\subsection{SHA Gromov boundaries}
\label{SHA}
Recall, cf.\ \cite{Dav} (see also \cite{O-ib7scg}), that a metric space $Z$ is \emph{strongly hereditarily aspherical} (shortly \emph{SHA}) if it can be embedded
in the Hilbert cube $Q$ in such a way that for each $\epsilon> 0$ there exists an $\epsilon$--covering $\mathcal U$ of $Z$ by open subsets of $Q$, where the union of any subcollection of elements of $\mathcal U$ is
aspherical. This notion was introduced by Robert J. Daverman \cite{Dav} and its significance follows from the fact that a cell-like map defined on a strongly hereditarily aspherical
compactum does not raise dimension.

\begin{lem}
\label{pi-1asf}
Let $v$ be a vertex of a weakly systolic $SD_2^{\ast}(7)$ complex $X$ whose links are $SD_2^{\ast}$ complexes and let $\pi_v^2=\pi_v \circ \pi_v \colon S_{i+2}(v,X) \to S_i(v,X)$ be the projection (as in Lemma \ref{newpi}). Then, for every subcomplex $L$ of $S_{i}(v,X)$, its preimage $(\pi_v^2)^{-1}(L)\subseteq S_{i+2}(v,X)$ is aspherical.
\end{lem}
\dow
The proof follows the one of \cite[Lemma 3.4]{O-ib7scg}. Instead of \cite[Proposition 2.1]{O-ib7scg} and \cite[Theorem 2.5]{O-ib7scg} we use, respectively, Proposition \ref{fullasf} and Theorem \ref{logl}. Instead of \cite[Lemma 3.1]{O-ib7scg} we use Lemma \ref{strcon}.
\kon

\begin{tw}[SHA Gromov boundary]
\label{shatw}
Let $G$ be a group acting geometrically by automorphisms on
a weakly systolic $SD_2^{\ast}(7)$ complex $X$ whose links are $SD_2^{\ast}$ complexes. Then the Gromov boundary of $G$ is strongly hereditarily aspherical (SHA).
\end{tw}
\dow
The proof goes along the lines of the proof of \cite[Theorem 4.2]{O-ib7scg} applying \cite[Proposition 2.9]{O-ib7scg} to the inverse sequence $\lk (S_{2n}(v,X))', \pi_v^2 \rk$ (for some vertex $v$ of X; cf.\ Theorem \ref{bdryinv}) and using Lemma \ref{pi-1asf} above.
\kon

\rems
1) The only known up to now high dimensional Gromov hyperbolic groups with SHA boundaries were $7$--systolic groups; cf.\ \cite{O-ib7scg}. Theorem \ref{shatw} extends those results to some new systolic groups. For example, if $X$ is a simply connected locally $6$--large cubical complex, then, by Corollary \ref{cat-1cc}, its thickening $Th(X)$ is a weakly systolic $SD_2^{\ast}(k)$ complex (for every $\geqslant 6$) with $6$--large (by Lemma \ref{k-l.thick}) links. Thus the Gromov boundary of $X$ is SHA. Examples of groups acting geometrically on such complexes include right-angled Coxeter groups with $6$--large nerves (then the corresponding cubical complex is the Davis complex).

2) In \cite{OS}, we present a simple construction of highly dimensional groups acting on complexes with $SD_2^{\ast}(k)$ links, $k\geqslant 6$. Those complexes can be non-systolic, thus we get a priori new examples of highly dimensional hyperbolic groups with SHA boundary.

3) Jacek {\' S}wi{\polhk a}tkowski \cite{Sw-propi} introduced a property of \emph{pro-$\pi_1$--saturation} and proved that the Gromov boundary of a $7$--systolic group is pro-$\pi_1$--saturated. This seems to be stronger than SHA. It is likely that the Gromov boundary of a weakly systolic $SD_2^{\ast}(7)$ complex with $SD_2^{\ast}$ links is pro-$\pi_1$--saturated. In particular we believe that it is relatively easy to show e.g.\ that boundaries of right-angled Coxeter groups with $6$--large nerves are pro-$\pi_1$--saturated.


\end{document}